\documentclass{siamltex}
\usepackage{amsmath, amssymb, amsfonts}
\usepackage{esint}
\usepackage{graphicx, epsfig,verbatim}
\usepackage{subfigure,tabularx}
\usepackage{pstricks,pst-node,pst-text,pst-3d}
\usepackage{subfigure}
\usepackage{color} 

\newcommand{\be}{\begin{equation}}
\newcommand{\ee}{\end{equation}}

\newcommand{\ignore}[1]{}

\newcommand{\vm}[1]{\mathbf{#1}}
\graphicspath{{figures/}}

\newtheorem{thm}{Theorem.}
\newtheorem{cor}[thm]{Corollary}

\newtheorem{rmk}[thm]{Remark}

\newcommand{\bbR}{\mathbb{R}}
\newcommand{\vxi}{\boldsymbol{\xi}}

\newcommand{\tvx}{\tilde{\vm{x}}}

\title{\bf New Epitaxial Thin Film Models and numerical approximation}
\author{
Wenbin Chen\thanks{School of Mathematical Science, Fudan University, Shanghai, P.R.China, 200433, {\tt wbchen@fudan.edu.cn}.
Supported by the Natural Science Foundation of China (11171077), Key Project National Science Foundation of China (91130004), the Ministry of Education of China and
the State Administration of Foreign Experts Affairs of China under the 111 project grant (B08018).}
 \and
 Zhenhua Chen\thanks{School of Mathematical Science, Fudan University, Shanghai, P.R.China, 200433, {\tt chenzhenhua@fudan.edu.cn}.}
\and
 Jin Cheng\thanks{School of Mathematical Science, Fudan University, Shanghai, P.R.China, 200433, {\tt jcheng@fudan.edu.cn}.  Supported by the Natural Science Foundation of China (11331004).}
\and
Yanqiu Wang\thanks{Department of Mathematics, Oklahoma State University, Stillwater OK 74078, {\tt yanqiu.wang@okstate.edu}.}
 }

\begin{document}
\maketitle

\begin{abstract}
This paper concerns new continuum phenomenological model for epitaxial thin film growth
with three different forms of the Ehrlich-Schwoebel current.
Two of these forms were first proposed by Politi and Villain \cite{PV96} and then studied by Evans, Thiel and Bartelt \cite{Evans}.
The other one is completely new.
Following the techniques used in Li and Liu \cite{LL03}, we present rigorous analysis of the well-posedness, regularity and time stability
for the new model.
We also studied both the global and the local behavior of the surface roughness in the growth process.
The new model differs from other known models in that it features a linear convex part and a nonlinear concave part,
and thus by using a convex-concave time splitting scheme, one can naturally
build unconditionally stable semi-implicit numerical discretizations with linear implicit parts, which is much easier to implement than conventional models
requiring nonlinear implicit parts.
Despite this fundamental difference in the model,
numerical experiments show that the nonlinear morphological instability of the new model agrees well with results
of other models published in \cite{LL03}, which indicates that the new model correctly captures the essential morphological states
in the thin film growth process.
\end{abstract}

\pagestyle{myheadings}
\thispagestyle{plain}

\section{Introduction}
In epitaxial thin film growth,
the phenomenological  continuum evolution of film height $h(\vm{x}, t)$
at lateral position $\vm{x}\in \Omega\subset\bbR^2$ and time $t\in (0,T)$ is governed by the equation \cite{MK04}
\begin{eqnarray}\label{eveq}
\partial_t h = \zeta-\nabla\cdot \vm{J},
\end{eqnarray}
where $\zeta(\vm{x}, t)$ is a given function related to the deposition rate,
and $\vm{J}(\vm{x}, t)$ is the lateral mass current of adatoms across the film surface.
The current $\vm{J}$ consists of an equilibrium (EQ) part and a non-equilibrium (NE) part, denoted by
$\vm{J} = \vm{J}_{EQ} + \vm{J}_{NE}$.
For the equilibrium part, we adopt the linearized model of Mullins \cite{Mullins59} and set $\vm{J}_{EQ} = K_{EQ} \nabla (\Delta h)$,
where the constant $K_{EQ}\ge 0$ is usually very small.
The more interesting non-equilibrium surface current $\vm{J}_{NE}$ depicts the interaction of adatoms with surface steps, and here we follow the model
presented by Evans, Thiel and Bartelt \cite{Evans}:
$$
\vm{J}_{NE}= \vm{J}_{DF} + \vm{J}_{ES} + \vm{J}_{RELAX},
$$
where $\vm{J}_{DF}=-\gamma \nabla h$, with constant $\gamma\ge 0$, is the stabilizing downward funneling (DF) current;
$\vm{J}_{{ES}}$ is the de-stabilizing uphill Ehrlich-Schwoebel (ES) current, which will be discussed in further details later;
and $\vm{J}_{RELAX} = \kappa \nabla (\Delta h)$, with constant $\kappa\ge 0$, is a phenomenological relaxation current artificially added when $\vm{J}_{EQ}\approx 0$,
in order to counteract the increasingly violent unstable behavior caused by $\vm{J}_{ES}$.
Mathematically, one can combine $\vm{J}_{RELAX}$ with $\vm{J}_{EQ}$ to get
$\vm{J}_{EQ} + \vm{J}_{RELAX} = \epsilon^2 \nabla (\Delta h)$, with $\epsilon^2 = \kappa + K_{EQ}$.

Now let us examine $\vm{J}_{ES}$, which models the Ehrlich-Schwoebel (ES) effect.
The ES effect states that adatoms must overcome a higher energy barrier in order to attach to a step from an
upper terrace than from a lower terrace. Thus it favors an uphill current and consequently causes the formation
as well as steepening of mounds \cite{GF,LL03,RLS69,RE}.
Due to its nonlinear nature, the ES current brings interesting surface morphological instability,
but imposes difficulty upon the mathematical analysis.
To our knowledge, there exist three ES models which have been mathematically investigated in terms of well-posedness and properties of the solution:
\begin{enumerate}
\item Infinite ES barrier model proposed in \cite{Villain91} with ES current $\vm{J}_{I, ES} = \frac{\nabla h}{|\nabla h|^2}$;
\item Finite ES barrier model proposed in \cite{Johnson94} with ES current  $\vm{J}_{F,ES} =\frac{\nabla h}{1+|\nabla h|^2}$;
\item Finite ES barrier with slope selection model (see \cite{LL03}) with ES current  $\vm{J}_{FSS,ES} = (1-|\nabla h|^2)\nabla h$;
\end{enumerate}
where $|\cdot|$ stands for the Euclidean length of a vector.
Note that $\vm{J}_{FSS,ES}$ and $\vm{J}_{I, ES}$ can be viewed as approximations to $\vm{J}_{F,ES}$ when $|\nabla h|\ll 1$ and $|\nabla h|\gg 1$, respectively.
In \cite{LL03}, well-posedness and long time stability have been established for the two finite ES barrier models.
We point out that a main difference, from the mathematical point of view, between the infinite ES barrier model and the finite ES barrier models is that,
$\vm{J}_{I, ES}$ is not continuous at $\nabla h = \vm{0}$, while $\vm{J}_{F,ES}$ and $\vm{J}_{FSS,ES}$ remain continuous for all $\nabla h$.
This is why rigorous mathematical analysis in \cite{LL03} only works for the two finite ES barrier models.

In this paper, we are interested in a different form of ES current first proposed by Politi and Villain \cite{PV96},
and later studied by Evans, Thiel and Bartelt \cite{Evans}:
\begin{equation}\label{J1}
\vm{J} _{1,ES}=\alpha_1\left(\frac{\nabla h}{p+|\nabla h|}
-\frac{\nabla h}{q+|\nabla h|}\right),
\end{equation}
where $\alpha_1 > 0$ and $0< p < q$ are given parameters. 
There are more physical parameters involved in  the real models in \cite{Evans}, and we only try to describe one simple
but essential model here.
Physical meanings and practical choices of these parameters will be presented in Section \ref{Snumer}.
For now we focus on the mathematical properties of the ES current.

When $|\nabla h|\ll p$, by eliminating high order terms one has
$$
\vm{J} _{1,ES}= \alpha_1 \frac{q-p}{pq + (p+q)|\nabla h| + |\nabla h|^2} \nabla h
 \approx \left(\alpha_1\frac{q-p}{p}\right)\, \frac{\nabla h}{q + (p+q)|\nabla h|/p }.
$$
Thus we introduce a variation of the ES current
\begin{equation}\label{J2}
\displaystyle \vm{J} _{2,ES}
=\alpha_2\frac{\nabla h}{q+(p+q)|\nabla h|/p},
\end{equation}
with $\alpha_2 = \alpha_1 \frac{q-p}{p}$. To our knowledge, this model of the ES current is brand new.

Similarly, when $|\nabla h|\gg p$, one has
$$
\vm{J} _{1,ES}= \alpha_1 \frac{q-p}{p(q+2|\nabla h|) + (q-p)|\nabla h| + |\nabla h|^2} \nabla h
   \approx \alpha_1(q-p)\, \frac{\nabla h}{(q-p)|\nabla h| + |\nabla h|^2},
$$
because $p(q+2|\nabla h|) = p^2 + p(q-p) + 2p|\nabla h| \ll (q-p)|\nabla h| + |\nabla h|^2$.
This allows us to introduce another variation of the ES current
\begin{equation}\label{J3}
\displaystyle \vm{J} _{3,ES}
=\alpha_3\frac{\nabla h}{(q-p)|\nabla h| + |\nabla h|^2},
\end{equation}
with $\alpha_3 = \alpha_1(q-p)$.
A much simpler one-dimensional form of $\vm{J} _{3,ES}$ has been proposed and studied in \cite{PV96, Evans}.
We believe this is the first time that the multi-dimensional form of $\vm{J} _{3,ES}$ is presented.

The main purpose of this paper is to analyze mathematically the epitaxial thin film growth equation (\ref{eveq})
with ES currents $\vm{J} _{k,ES}$, for $k=1,2$ and $3$.
Note that $\vm{J} _{3,ES}$ is not continuous at $\nabla h = \vm{0}$, while $\vm{J} _{1,ES}$ and $\vm{J} _{2,ES}$
are continuous for all $\nabla h$.
In this sense, one may compare $\vm{J} _{3,ES}$ with the infinite ES current $\vm{J}_{I,ES}$. Similarly, $\vm{J} _{1,ES}$ is comparable to
the finite ES current without slope selection $\vm{J}_{F,ES}$, and $\vm{J} _{2,ES}$ is comparable to the finite ES current with slope selection $\vm{J}_{FSS,ES}$.
Later it shall become clear that the models $\vm{J} _{k,ES}$, for $k=1,2,3$, have built-in and significant differences from
$\vm{J}_{F,ES}$, $\vm{J}_{FSS,ES}$, and $\vm{J}_{I,ES}$ in the mathematical analysis. Though interestingly, numerical results
will show that they give very similar nonlinear morphological evolution processes, which is a good sign as they all model the same physical phenomenon.

For simplicity, let $\Omega$ be a rectangular domain and set the $\Omega$-periodic boundary condition on $h$.
Following the previous discussions, Equation (\ref{eveq}) equipped with initial and boundary conditions can be written as
\begin{equation} \label{eq:diffeq}
\begin{aligned}
&\partial_t h = \zeta+\gamma\Delta h-\epsilon^2\Delta^2 h- \nabla\cdot \vm{J}_{ES}, \qquad\textrm{in } \Omega\times (0,T],\\
&h(\cdot,t) \textrm{ is }\Omega\textrm{-periodic for all }t\in [0,T], \\
&h(\vm{x},0) = h_0(\vm{x}) \textrm{ for all }\vm{x}\in \Omega,
\end{aligned}
\end{equation}
where $\vm{J}_{ES}$ is chosen from $\vm{J}_{k, ES}$ for $k=1,2,3$.
Here and throughout the rest of the paper, we shall only use subscript $k$ when individual features of the ES current from different models
are needed. Otherwise, the ES current will simply be denoted as $\vm{J}_{ES}$, which can be any applicable
choice from $\vm{J}_{k, ES}$ for $k=1,2,3$.
Because of the singularity of $\vm{J} _{3,ES}$ at $\nabla h = \vm{0}$, rigorous mathematical analysis in the rest of this paper will only be performed
for $\vm{J} _{1,ES}$ and $\vm{J} _{2,ES}$. Though we still keep $\vm{J} _{3,ES}$ as an alternative option whenever applicable.
For compatibility purpose, obviously $h_0$ and $\zeta$ should also be $\Omega$-periodic.

Next, we introduce a surface roughness indicator and also argue that it suffices to study (\ref{eq:diffeq})
under the assumption that $\zeta$ and $h_0$ are mean value free on $\Omega$.
For simplicity, denote $\fint_{\Omega} f\, dx = \frac{1}{|\Omega|} \int_{\Omega} f\, dx$ for any given function $f$.
Define the average height function $\bar{h}(t)$,
for $t\in [0,T]$ by $\bar{h} = \fint_{\Omega} h\, dx$.
Then, the surface roughness is estimated by \cite{LL04}
$$
\omega(t) = \sqrt{\fint_{\Omega} |h(\vm{x},t) - \bar{h}(t)|^2\, dx },\qquad\textrm{for } t\in [0,T].
$$

Similarly, denote $\bar{\zeta} = \fint_{\Omega} \zeta \, dx$.
By taking the average integral of the differential equation in (\ref{eq:diffeq}) and using the $\Omega$-periodic boundary condition of $h$,
it is clear that $\bar{h}$ satisfies an ordinary differential equation
$$\partial_t \bar{h} = \bar{\zeta}.
$$
Subtracting this equation from (\ref{eq:diffeq}) and noticing that any spatial derivative of $\bar{h}$ is $0$,
one immediately gets
$$
\partial_t (h-\bar{h}) = (\zeta-\bar{\zeta}) + \gamma\Delta(h-\bar{h}) - \epsilon^2 \Delta^2(h-\bar{h}) - \nabla\cdot \vm{J}_{ES}(\nabla(h-\bar{h})).
$$
In other words, $h-\bar{h}$ satisfies Equation (\ref{eq:diffeq}), with $\zeta$ in the right-hand side replaced by $\zeta-\bar{\zeta}$,
the mean-value free component of $\zeta$. Thus studying $h-\bar{h}$ is equivalent to studying $h$ with the assumption
that $\zeta$ and $h_0$ are mean value free.
In this case the surface roughness indicator becomes
$$
\omega(t) = \sqrt{\fint_{\Omega} |h(\vm{x},t)|^2\, dx } = \frac{1}{\sqrt{|\Omega|}} \|h(\cdot, t)\|_{L^2(\Omega)},\qquad\textrm{for } t\in [0,T].
$$
Throughout the rest of this paper, we shall assume that $\zeta$ and $h_0$ are mean value free on $\Omega$,
and consequently so is $h$.
Note that a typical example is $\zeta \equiv 0$.

Using techniques similar to those in \cite{LL03}, i.e., the well-known Galerkin approximation and compactness argument approach
of Lions \cite{Lions},
we will establish the existence, uniqueness, and regularity of the weak solution to (\ref{eq:diffeq}).
The theoretical proof, although standard, relies heavily on particular properties of the ES current $\vm{J}_{ES}$.
One of the main contribution of this paper is to establish these properties for $\vm{J}_{k,ES}$, with $k=1,2$, and
part of the properties for $\vm{J}_{3,ES}$.

We will also establish global and local bounds for the surface roughness $\omega(t)$. The epitaxial thin film growth is in general
a coarsening process, for which $\omega(t)$ is an important indicator. In the early stage of the growth,
a typical rough-smooth-rough pattern \cite{GZ98,LL03} is often observed.
Hence theoretical and numerical study of $\omega(t)$ is important to the understanding of the surface morphological evolution.
Besides the roughness indicator, the growth process is always energy driven in the sense that the dynamics is the gradient flow of
a certain energy functional \cite{Evans10,Evans,RVK06,RVK02,RVK03}. We will show that the energy functional remains non-increasing
with our ES current models, when the deposition rate $\zeta \equiv 0$.

Numerical discretization will be done using the convex-concave splitting technique.
This technique was first proposed by Eyre to solve the Cahn-Hilliard and Allen-Cahn equations \cite{Eyre98}. 
Its main idea is to treat the ``convex" part implicitly and the ``concave" part explicitly in the time discretization.
From another point of view, this is equivalent to solving a minimization problem of a strictly convex and coercive functional
known as the modified energy functional \cite{WW09,CW,HW09}.
Eyre's convex-concave splitting scheme is first-order accurate in time and unconditionally stable.
Later, higher order time schemes have been constructed using the similar idea \cite{HW09,SW11}.
For thin film epitaxial growth with ES currents $\vm{J}_{F,ES}$ and $\vm{J}_{FSS,ES}$, the convex-concave splitting
inevitably generates a nonlinear convex part \cite{CW, CWmixed}, though alternative schemes with linear explicit parts
can derived using other techniques \cite{WS, Xu06}. A significant advantage of the new ES models $\vm{J}_{k,ES}$, for $k=1,2,3$,
is that, they naturally generate linear convex parts and nonlinear concave parts in the splitting.
Hence the direct application of the convex-concave splitting technique will result in a linear problem to solve at
each time step.
Spatial discretization is done by a Fourier spectral Galerkin method.

The rest of the article is organized as follows. In Section
\ref{Smodel}, we establish the existence, uniqueness, and regularity results of the weak solution
to the model problem. In addition, bounds of the roughness indicator $\omega(t)$  and analysis of the energy functional will also be given in this section.
In Section
\ref{Sprop},  a semi-implicit fully-discrete numerical schemes using the convex-concave splitting technique
is presented. We show that the scheme is unconditionally stable. Convergence rate is also proved.
In Section \ref{Snumer}, we present numerical results which show similar morphological instability
as results given in \cite{LL03}.

\section{Well-posedness of the model problem}\label{Smodel}
In this section, we study the well-posedness of Equation (\ref{eq:diffeq}).
As mentioned earlier, rigorous analysis will only be done when the ES current $\vm{J}_{ES}$ is taken to be
either $\vm{J}_{1,ES}$ or $\vm{J}_{2,ES}$. We shall first prove a few properties of the ES current in Section \ref{subsec:Smodel1},
in which we conveniently use a subscript $k=1,2$ to denote whether $\vm{J}_{ES}$ is taken to be $\vm{J}_{1,ES}$ or $\vm{J}_{2,ES}$,
as the proof depends on the individual definitions of $\vm{J}_{k,ES}$.
It is worth to point out that $\vm{J}_{3,ES}$ also possesses some similar properties, especially the most important
convex-concave splitting one. This is why we do not want to completely leave it out,
and the properties of $\vm{J}_{3,ES}$ will be mentioned in a separate remark.
After these properties are established by $k$-specific proofs, for simplicity we will drop the subscript $k$
when the analysis does not depend on $k$.

\subsection{Properties of the function $\vm{J}_{ES}$} \label{subsec:Smodel1}

We start from $\vm{J}_{k,ES}$ for $k=1,2$.
Note that $\vm{J}_{k,ES}$ depends solely on $\nabla h$. It is convenient to view them as functions
$\vm{J}_{k,ES}(\vm{m})$ taking values at $\vm{m} = \nabla h$.
Moreover, by definition, we can write $\vm{J}_{k,ES}(\vm{m}) = \Phi_k(|\vm{m}|) \vm{m}$, where
$$
\Phi_1(s) = \alpha_1 \frac{q-p}{(p+s)(q+s)}, \quad\textrm{and}\quad
\Phi_2(s) = \alpha_2 \frac{1}{q+(p+q)s/p},
$$
for all $s\ge 0$.

\smallskip
\begin{lemma} \label{lem:PhiBounds}
For all $s\ge 0$ and $k=1,2$, one has
$$
0< \Phi_k(s) \le C,\qquad  -C\le \Phi_k'(s) <0,
$$
where $C$ is a positive general constant depending only on $\alpha_k$, $p$, and $q$.
\end{lemma}
\begin{proof}
The bounds for $\Phi_k(s)$, $k=1,2$ are obvious, and the bounds for $\Phi_k'(s)$ follows immediately from
$$
\begin{aligned}
\Phi_1'(s) &= -\alpha_1\frac{2(q-p)((p+q)/2+s)}{(p+s)^2(q+s)^2} \ge -\alpha_1\frac{2(q-p)}{(p+s)^2(q+s)} \ge -\alpha_1\frac{2(q-p)}{p^2q}, \\
\Phi_2'(s) &= -\alpha_2\frac{(p+q)/p}{(q+(p+q)s/p)^2} \ge -\alpha_2\frac{p+q}{pq^2}.
\end{aligned}
$$
\end{proof}

Another important observation is that both $\vm{J}_{k,ES}(\vm{m})$, for $k=1,2$, are gradient fields.
Indeed, define functions $G_k: \: \bbR^2\to \bbR$ by
$$
\frac{1}{\alpha_1}G_1(\vm{m}) = p\ln(p+|\vm{m}|) - q\ln(q+|\vm{m}|) \quad\textrm{and}\quad
\frac{1}{\alpha_2}G_2(\vm{m}) = -\frac{p}{p+q}|\vm{m}| + \frac{p^2q}{(p+q)^2} \ln \left( \frac{pq}{p+q} + |\vm{m}| \right),
$$
for all $\vm{m}\in\bbR^2$.
Now we examine the derivatives of $G_k(\vm{m})$ with respect to variable $\vm{m}$.
In order to distinguish such derivatives with the spatial derivatives, we use $\nabla_F G_k(\vm{m})$
and $\nabla_F^2 G_k(\vm{m})$ to denote the gradient and the Hessian of $G_k(\vm{m})$ with respect to $\vm{m}$,
while reserving the notation $\nabla$ and $\nabla^2$ for gradient and Hessian with respect to the spatial variable $\vm{x}$.

\smallskip
\begin{lemma} \label{lem:gradHessianGk}
For $k=1,2$, one has $G_k\in C^2(\bbR^2)$. Their gradients satisfy
$$
\nabla_F G_k(\vm{m}) = - \Phi_k(|\vm{m}|) \vm{m},
$$
and their Hessians satisfy
\begin{equation} \label{eq:Hessians}
\begin{aligned}
\nabla_F^2 G_1(\vm{m}) &= -\alpha_1\frac{q-p}{(p+|\vm{m}|)(q+|\vm{m}|)}I + \alpha_1\left(\frac{1}{(p+|\vm{m}|)^2} - \frac{1}{(q+|\vm{m}|)^2}\right) \frac{\vm{m} \otimes \vm{m}}{|\vm{m}|}, \\
\nabla_F^2 G_2(\vm{m}) &= -\alpha_2\frac{p}{pq+(p+q)|\vm{m}|} I + \alpha_2\frac{p(p+q)}{(pq+(p+q)|\vm{m}|)^2} \, \frac{\vm{m} \otimes \vm{m}}{|\vm{m}|},
\end{aligned}
\end{equation}
where $I$ is the $2\times 2$ identity matrix and $\vm{m} \otimes \vm{m}$ is a $2\times 2$ matrix defined by $\vm{m}\vm{m}^T$,
in which $\vm{m}$ is considered as a column vector.
\end{lemma}
\begin{proof}
The proof is elementary. One only needs to use the fact that
$\nabla_F |\vm{m}| = \vm{m}/|\vm{m}|$ and $\nabla_F \frac{\vm{m}}{c+|\vm{m}|} = \frac{1}{c+|\vm{m}|}I - \frac{1}{(c+|\vm{m}|)^2}\,\frac{\vm{m} \otimes \vm{m}}{|\vm{m}|}$, for all $c> 0$,
to compute $\nabla_F G_k(\vm{m})$ and $\nabla_F^2 G_k(\vm{m})$.
\end{proof}

\smallskip
\begin{cor} \label{cor:Lipschitz}
We clearly have $\vm{J}_{k,ES}\in C^1(\bbR^2)$ for $k=1,2$, and hence they are locally Lipschitz.
\end{cor}
\smallskip

Next we shall discuss the convex splitting of functions $G_k(\cdot)$, for $k=1,2$.
We say a function is convex if its Hessian matrix is positive semi-definite everywhere,
and concave if its Hessian matrix is negative semi-definite everywhere. It is not hard to see that
\smallskip
\begin{lemma}  \label{lem:GkConcave}
For all $\chi_1 \ge 2\alpha_1\frac{q-p}{pq}$ and $\chi_2\ge 0$, the function $G_k(\vm{m}) -\frac{1}{2}\chi_k|\vm{m}|^2$, for $k=1,2$, is concave.
\end{lemma}
\begin{proof}
Note that the two eigenvalues of matrix $\vm{m}\otimes \vm{m}$ are $0$ and $|\vm{m}|^2$.
By (\ref{eq:Hessians}), it is clear that the two eigenvalues of $\nabla_F^2 G_1(\vm{m})$ are
$$
\begin{aligned}
\lambda_1 &= -\alpha_1\frac{q-p}{(p+|\vm{m}|)(q+|\vm{m}|)} <0, \\
\lambda_2 &= -\alpha_1\frac{q-p}{(p+|\vm{m}|)(q+|\vm{m}|)} + \alpha_1\left(\frac{1}{(p+|\vm{m}|)^2} - \frac{1}{(q+|\vm{m}|)^2}\right) |\vm{m}| < \alpha_1\left(\frac{1}{(p+|\vm{m}|)^2} - \frac{1}{(q+|\vm{m}|)^2}\right) |\vm{m}| \\
          &= \alpha_1\frac{(q-p)(p+q+2|\vm{m}|)|\vm{m}|}{(p+|\vm{m}|)^2(q+|\vm{m}|)^2} = 2 \alpha_1\left(\frac{q-p}{(p+|\vm{m}|)(q+|\vm{m}|)}\right)\, \left(\frac{(p+q)/2 + |\vm{m}|}{q+|\vm{m}|}\right)\, \left(\frac{|\vm{m}|}{p+|\vm{m}|}\right) \\
          & < 2\alpha_1\frac{q-p}{pq} \le \chi_1.
\end{aligned}
$$
This, combined with the fact that $\nabla_F^2 (\frac{1}{2}\chi_1|\vm{m}|^2) = \chi_1 I$, implies that $G_1(\vm{m}) -\frac{1}{2}\chi_1|\vm{m}|^2$ is concave.
Similarly, the two eigenvalues of $\nabla_F^2 G_2(\vm{m})$ are
$$
\begin{aligned}
\lambda_1 &= -\alpha_2\frac{p}{pq+(p+q)|\vm{m}|} < 0, \\
\lambda_2 &= -\alpha_2\frac{p}{pq+(p+q)|\vm{m}|} + \alpha_2\frac{p(p+q)}{(pq+(p+q)|\vm{m}|)^2}|\vm{m}| = -\alpha_2\frac{p^2q}{(pq+(p+q)|\vm{m}|)^2} <0.
\end{aligned}
$$
Hence $G_2(\vm{m}) -\frac{1}{2}\chi_2|\vm{m}|^2$ is concave. This completes the proof of the lemma.
\end{proof}
\smallskip

\begin{cor} \label{cor:convexsplitting}
The functions $G_k(\cdot)$, for $k=1,2$, have the convex-concave splitting $G_k(\cdot) = G_{k,+}(\cdot) + G_{k,-}(\cdot)$,
where the convex and the concave parts are defined, respectively, by
$$
G_{k,+}(\vm{m}) = \frac{1}{2}\chi_k|\vm{m}|^2,\qquad G_{k,-}(\vm{m}) = G_k(\vm{m}) -\frac{1}{2}\chi_k|\vm{m}|^2,
$$
for all $\chi_1 \ge 2\alpha_1\frac{q-p}{pq}$ and $\chi_2\ge 0$.
\end{cor}
\medskip

The convex splitting and its properties are essential in theoretical analysis and the constructing of numerical schemes.
In \cite{LL03, CW}, several bounds of the convex splitting for ES currents $\vm{J}_{F,ES}$ and $\vm{J}_{FSS,ES}$
have been proved. Next, we shall prove similar bounds for ES currents $\vm{J}_{k,ES}$, with $k=1,2$.
\smallskip
\begin{lemma} \label{lem:GkBounds1}
For any $\vm{m}\in \bbR^2$ and $k=1,2$, we have
$$
|G_k(\vm{m})|\le C(1+|\vm{m}|)\qquad\textrm{and}\qquad |\nabla_F G_k (\vm{m})| \le C,
$$
where $C$ is a general constant depending only on $\alpha_k$, $p$ and $q$.
Moreover, all eigenvalues of $\nabla_F^2 G_k(\vm{m})$ have absolute values bounded by $C$.
In other words, the matrix $2$-norm of $\nabla_F^2 G_k(\vm{m})$, denoted by $|\nabla_F^2 G_k(\vm{m})|$, has bound
$$
|\nabla_F^2 G_k(\vm{m})| \le C.
$$
\end{lemma}
\begin{proof}
The proof for $|G_k(\vm{m})|\le C(1+|\vm{m}|)$ follows immediately from the fact that $\ln(1+s)\le s$ for $s\ge 0$,
while the proof of $|\nabla_F G_k (\vm{m})| \le C$ is elementary by Lemma \ref{lem:gradHessianGk} and the definition of $\Phi_k(\cdot)$.
Finally, the claim about eigenvalues of $\nabla_F^2 G_k(\vm{m})$ follows from the computation of these eigenvalues in the proof of Lemma \ref{lem:GkConcave}.
\end{proof}

\smallskip
\begin{cor} \label{cor:GkBounds}
By Lemma \ref{lem:GkBounds1} and Corollary \ref{cor:convexsplitting}, one immediately have
$$
\begin{aligned}
|G_{k,+}(\vm{m})| + |G_{k,-}(\vm{m})| & \le C(1+|\vm{m}|^2), \\
|\nabla_F G_{k,+}(\vm{m})| + |\nabla_F G_{k,-}(\vm{m})| & \le C(1+|\vm{m}|), \\
|\nabla_F^2 G_{k,+}(\vm{m})| + |\nabla_F^2 G_{k,-}(\vm{m})| & \le C,
\end{aligned}
$$
for  any $\vm{m}\in \bbR^2$ and $k=1,2$.
\end{cor}
\smallskip

In addition, we also have the following lemma:

\smallskip
\begin{lemma} \label{lem:GkBounds2}
For any constant $\beta>0$, there exists a $C_{\beta}>0$ such that
$$
G_k(\vm{m}) \ge -\beta|\vm{m}|^2 - C_{\beta},\qquad\textrm{for all }\vm{m}\in \bbR^2\textrm{ and }k=1,2.
$$
\end{lemma}
\begin{proof}
By Lemma \ref{lem:GkBounds1} and the Young's inequality, one has
$$
G_k(\vm{m}) \ge -C(|\vm{m}|+1) \ge -C|\vm{m}|-C \ge -\beta|\vm{m}|^2 - \frac{C^2}{4\beta} - C,
$$
where $C$ is a positive constant. This completes the proof of the lemma.
\end{proof}
\smallskip

We have so far stated all properties of $\vm{J}_{k,ES}$ needed in the analysis of Equation (\ref{eq:diffeq}).
Note that these properties hold for both $k=1$ and $k=2$.
It turns out that $\vm{J}_{3,ES}$, although not continuous at $\nabla h = 0$, also satisfy some of these properties.
We summarize it in the following remark:

\begin{rmk} \label{rem:J3}
Similar analysis shows that the same properties as presented in this subsection hold for $\vm{J}_{3,ES}$ as long as $\nabla h$ stays away from $\vm{0}$.
Below are the details. Define
$$
\Phi_3(s) = \alpha_3 \frac{1}{s^2+(q-p)s}\quad\textrm{and}\quad \frac{1}{\alpha_3}G_3(\vm{m}) = -\ln (q-p+|\vm{m}|).
$$
Then, one has $\Phi_3 \in C^1(\bbR^+)$, $G_3\in C(\bbR^2)\cap C^2(\bbR^2\backslash\{\vm{0}\})$,
and for all $\vm{m}\in \bbR^2\backslash\{\vm{0}\}$,
$$
\begin{aligned}
\nabla_F G_3(\vm{m}) &= - \Phi_3(|\vm{m}|) \vm{m}, \\[2mm]
\nabla_F^2 G_3(\vm{m}) &= -\alpha_3\frac{1}{(q-p)|\vm{m}|+|\vm{m}|^2}I + \alpha_3 \frac{q-p+2|\vm{m}|}{\left((q-p)|\vm{m}|+|\vm{m}|^2\right)^2} \, \frac{\vm{m} \otimes \vm{m}}{|\vm{m}|}.
\end{aligned}
$$
Moreover, $G_3$ has the convex-concave splitting $G_3 = G_{3,+} + G_{3,-}$ where
$$
G_{3,+}(\vm{m}) = \frac{1}{2} \chi_3|\vm{m}|^2,\qquad G_{3,-}(\vm{m}) = G_3(\vm{m}) - \frac{1}{2} \chi_3|\vm{m}|^2,
$$
for all $\chi_3 \ge \frac{\alpha_3}{(q-p)^2}$.
When $s$ or $|\vm{m}|$ stays away from $0$, $\vm{J}_{3,ES}$ and $G_3$ have similar bounds as
in Lemmas \ref{lem:PhiBounds}, \ref{lem:GkBounds1}, \ref{lem:GkBounds1}, and Corollary \ref{cor:GkBounds},
but not when $s\to 0$ or $|\vm{m}|\to 0$.
\end{rmk}

\subsection{Weak solution to Equation (\ref{eq:diffeq})}
Due to the unboundedness of $\vm{J}_{3,ES}$ and $G_3$ mentioned in Remark \ref{rem:J3},
the analysis from here to the end of Section \ref{Sprop} only works for $\vm{J}_{k,ES}$, with $k=1,2$.
Using lemmas and corollaries proved in Section \ref{subsec:Smodel1}, we no longer need to distinguish
between $k=1$ and $k=2$ in the analysis to be given.
Therefore the subscript $k$ will be dropped for simplicity, i.e., without special mentioning, $\vm{J}_{ES}$, $\Phi(\cdot)$, $\chi$ and $G(\cdot)$ will
be used with definitions taken to be either for $k=1$ or $k=2$. Also, the convex splitting of $G(\cdot)$
defined in Corollary \ref{cor:convexsplitting} will simply be denoted by $G_+(\cdot)$ and $G_-(\cdot)$.
Occasionally, the case $k=3$ will be discussed individually in remarks.

In this subsection,
we define what is a weak solution to Equation (\ref{eq:diffeq}) and establish the existence, uniqueness as well as the regularity results of the weak solution.
The analysis follows exactly the same framework presented in \cite{LL03}, i.e., Lions method \cite{Lions} of
first constructing a semi-discrete Galerkin spectral approximation and then proving its convergence using a compactness argument,
as this is currently the most efficient approach for the given problem.
However, due to the different properties of $\vm{J}_{ES}$, there are still many essential differences
between our analysis and the one in \cite{LL03}, mainly in the proof of some inequalities.
Thus we still present the entire proof for completeness, although readers may find the majority of notation
and analysis are just borrowed from \cite{LL03}.

We first introduce the weak formulation of (\ref{eq:diffeq}).
Denote by $W^{m,r}_{per}(\Omega)$, for $m\ge 0$ and $1\le r\le \infty$ the $\Omega$-periodic Sobolev space with indices $m$ and $r$.
When $m=0$ and $r<\infty$, the space $W^{0,r}_{per}(\Omega)$ is simply the Lebesgue space $L^r(\Omega)$.
When $m\ge 1$ and $r=2$, the space $W^{m,2}_{per}(\Omega)$
is a Hilbert space and is also denoted by $H^m_{per}(\Omega)$.
For simplicity, denote by $\|\cdot\|$ the $L^2(\Omega)$ norm, while other Sobolev norms shall be explicitly specified in subscripts,
for example $\|\cdot\|_{H^1(\Omega)}$ and $\|\cdot\|_{L^2(0,T; H^2(\Omega))}$ --
note that in terms of norms there is no difference between $W^{m,r}(\Omega)$ and $W^{m,r}_{per}(\Omega)$ and hence the $per$ is omitted.
For $m<0$, denote by $H^m_{per}(\Omega)$ the dual space of $H^{-m}_{per}(\Omega)$.
Then the weak problem for Equation (\ref{eq:diffeq}) can be written as:
{\it Find $h$, in a proper space to be specified later, such that for all $t\in (0,T)$}
\begin{equation} \label{eq:weakeq}
\begin{aligned}
&\langle\partial_t h, \phi\rangle + a(h,\phi) = \langle\zeta,\phi\rangle,
\qquad\textrm{for all }\phi\in H^2_{per}(\Omega), \\
\textrm{with the form}\quad
a(h,\phi) &\triangleq \gamma \langle\nabla h,\nabla \phi\rangle + \epsilon^2 \langle\Delta h,\Delta \phi\rangle - \langle\Phi (|\nabla h|) \nabla h, \nabla\phi\rangle,
\end{aligned}
\end{equation}
where $\langle\cdot,\cdot\rangle$ denotes the duality pair, or the $L^2$ inner-product on $\Omega$ if both parties involved lie at least in $L^2(\Omega)$.
Lemma \ref{lem:PhiBounds} states that $0<\Phi (|\nabla h|) \le C$, thus
the nonlinear term $\langle\Phi (|\nabla h|) \nabla h, \nabla\phi\rangle$ in (\ref{eq:weakeq}) is well-defined
as long as $h$ and $\phi$ are in $H^1_{per}(\Omega)$.

\smallskip
\begin{definition}
We say $h:\:\Omega\times [0,T]\to \bbR$ is a weak solution to (\ref{eq:diffeq}) if it satisfies
\begin{enumerate}
\item $h\in L^2(0,T; H^2_{per}(\Omega))$ and $\partial_t h\in L^2(0,T; H^{-2}_{per}(\Omega))$;
\item Function $h$ satisfies the weak formulation (\ref{eq:weakeq}) almost everywhere for $t\in (0,T)$;
\item $h(\cdot,0) = h_0(\cdot)$ almost everywhere in $\Omega$.
\end{enumerate}
\end{definition}
\smallskip

\subsubsection{Semi-discrete Galerkin spectral approximation}
Here we define the semi-discrete Galerkin spectral approximation to (\ref{eq:weakeq}).
For any given $\vm{x}\in\Omega=(0,L_1)\times (0,L_2)$, denote $\tvx = 2\pi [x_1/L_1,\,x_2/L_2]^T \in (0,2\pi)\times (0,2\pi)$.
For a given positive integer $N$,
define the index space $\mathcal{I}_N = \{\vxi\in\mathbb{Z}^2\textrm{ with } 0\le \xi_1,\xi_2\le N\textrm{ and }\vxi\neq \vm{0}\}$
and a discrete space on $\Omega$ by
$$
H_N = span \{1, \,\cos \vxi\cdot\tvx,\, \sin \vxi\cdot\tvx, \textrm{ for all }\vxi\in\mathcal{I}_N \}.
$$
The space $H_N$ is $\Omega$-periodic.
Note that the spanning set of $H_N$ also forms an orthogonal basis for $H_N$ under the $L^2(\Omega)$ inner-product.
After proper ordering and normalizing, we get an orthonormal basis denoted by $\{\phi_i,\textrm{ for }i=1,\ldots,\Xi(N)\}$, where $\Xi(N) = \dim{H_N}$.
Denote by $P_N$ the $L^2$ projection onto $H_N$.
We will seek the $N$-th semi-discrete Galerkin spectral approximation to Equation (\ref{eq:weakeq}) in the space $H_N$ as following:
{\it Find $h_N = \sum_{i=1}^{\Xi(N)} h_{N,i}(t) \phi_{i}$ satisfying $h_N(\cdot,0) = P_N h_0(\cdot)$ in $\Omega$ and }
\begin{equation} \label{eq:Spectraleq}
\langle \partial_t h_N, \phi\rangle + a(h_N, \phi) = \langle\zeta,\phi \rangle,
\qquad\textrm{for all }\phi\in H_N,\, t\in (0,T].
\end{equation}
We point out that since $1\in H_N$, the operator $P_N$ maps mean value free functions to mean value free functions.
By setting $\phi=1$ in (\ref{eq:Spectraleq}), one has $\partial_t \bar{h}_N = 0$.
Combining the above, we know that $h_N$, if existing, is mean value free for all $t\in [0,T]$ as long as
$\zeta$ and $h_0$ are mean value free.

Before establishing the well-posedness of the Galerkin spectral approximation (\ref{eq:Spectraleq}), we first state two technique lemmas from \cite{LL03},
with a little extra obvious facts.
\smallskip
\begin{lemma} \label{lem:PerLaplacian}
For all $\phi\in H^2_{per}(\Omega)$, one has
$$
\|\nabla \phi\|^2 \le \|\phi\|\, \|\Delta\phi\|\qquad\textrm{and}\qquad \sum_{i,j = 1}^2 \|\partial_{x_ix_j} \phi\|^2 = \|\Delta\phi\|^2.
$$
Moreover, if $\phi$ is mean value free on $\Omega$, by the Poincar\'{e} inequality one has
$$
C\|\phi\|\le \|\nabla \phi\|\le \frac{1}{C}\|\Delta \phi\|,
$$
where $C$ is a positive general constant depending only on $\Omega$, and consequently $\|\phi\|_{H^2(\Omega)} \le C \|\Delta \phi\|$.
\end{lemma}

\begin{lemma} \label{lem:PNapproximation}
For any integer $m\ge 0$ and $\phi\in H^m_{per}(\Omega)$, one has
$$
\|P_N\phi\|_{H^m(\Omega)} \le \|\phi\|_{H^m(\Omega)},\qquad\textrm{and}\qquad P_N\phi \stackrel{N\to\infty}{\xrightarrow{\hspace{0.8cm}}} \phi \textrm{ strongly in }H^m_{per}(\Omega).
$$
Moreover, a direct calculation using Fourier series shows that
$$
\|\phi - P_N\phi\|_{H^j(\Omega)} \le C N^{-(m-j)} \|\phi\|_{H^m(\Omega)},\qquad \textrm{for }0\le j \le m,
$$
where $C$ is a positive general constant.
\end{lemma}

Next we prove the existence, uniqueness, and regularity of the Galerkin spectral approximation.
\smallskip
\begin{lemma} \label{lem:GalerkinSolution}
Assume that $h_0\in L^2(\Omega)$ and $\zeta \in L^{2}(0,T; H^{-2}_{per}(\Omega))$,
then for each integer $N\ge 1$, there exists a unique semi-discrete Galerkin spectral approximation $h_N$ satisfying (\ref{eq:Spectraleq}).
The solution $h_N$ has bound
\begin{equation} \label{eq:hNbound1}
\|h_N\|_{L^{\infty}(0,T; L^2(\Omega))} + \|h_N\|_{L^2(0,T; H^2(\Omega))} \le C(\|h_0\|, \|\zeta\|_{L^2(0,T; H^{-2}_{per}(\Omega))}),
\end{equation}
where $C(\|h_0\|, \|\zeta\|_{L^2(0,T; H^{-2}_{per}(\Omega))})$ is a positive constant depending on $\|h_0\|$ and $\|\zeta\|_{L^2(0,T; H^{-2}_{per}(\Omega))}$.
Moreover, if $h_0\in H^2_{per}(\Omega)$ and $\zeta \in L^{2}(0,T; L^2(\Omega))$, then we also have the following bound
\begin{equation} \label{eq:hNbound2}
\|h_N\|_{L^{\infty}(0,T; H^2(\Omega))} + \|h_N\|_{L^{2}(0,T; H^4(\Omega))} + \|\partial_t h_N\|_{L^{2}(0,T; L^2(\Omega))}  \le C(\|h_0\|_{H^2(\Omega)}, \|\zeta\|_{L^{2}(0,T; L^2(\Omega)}),
\end{equation}
where $C(\|h_0\|_{H^2(\Omega)}, \|\zeta\|_{L^{2}(0,T; L^2(\Omega)})$ is a positive constant depending on
$\|h_0\|_{H^2(\Omega)}$ and $\|\zeta\|_{L^{2}(0,T; L^2(\Omega)}$.
\end{lemma}
\begin{proof}
We follow the proof of Theorem 4.1 in \cite{LL03}, with some modifications on terms involving the ES current $\vm{J}_{ES}$.
By setting $\phi = \phi_j$ for $j=1,\ldots,\Xi(N)$ in Equation (\ref{eq:Spectraleq}) and
using the orthogonality of basis functions, we get a system of ordinary differential equations
\begin{equation} \label{eq:spectralODE}
\partial_t h_{N,j}(t) = f_j(\zeta(t), h_{N,1}(t),\ldots, h_{N,\Xi(N)}(t)),\quad\textrm{for all }j=1,\ldots,\Xi(N).
\end{equation}
Condition $h_N(\cdot,0) = P_N h_0(\cdot)$ actually sets the initial condition $h_{N,j}(0) = \langle h_0,\phi_j\rangle$, $j=1,\ldots,\Xi(N)$ for System (\ref{eq:spectralODE}).
A standard procedure to prove global existence and uniqueness of the solution to (\ref{eq:spectralODE}) is to first get local existence and uniqueness by the Picard-Lindel\"{o}f theorem,
i.e., by showing that $f_j$ are locally Lipschitz, and then prove that the solution is bounded for $t$ up to any given $T_N\le T$. We first argue that all $f_j$ are locally Lipschitz.
This indeed follows immediately from Corollary \ref{cor:Lipschitz} and the fact that composition, summation, and product of locally Lipschitz functions are also locally Lipschitz.

Now by the Picard-Lindel\"{o}f theorem, System (\ref{eq:spectralODE}) admits a unique local solution for $t$ from $0$ up to a $T_N$.
By setting $\phi=h_N(t)$ in Equation (\ref{eq:Spectraleq}) and using lemmas \ref{lem:PhiBounds}, \ref{lem:PerLaplacian} and the Young's inequality, one has
$$
\begin{aligned}
\frac{1}{2} \frac{d}{dt} \|h_N\|^2 + \gamma \|\nabla h_N\|^2 +\epsilon^2 \|\Delta h_N\|^2 &= \int_\Omega \Phi(|\nabla h_N|) |\nabla h_N|^2 \, dx + \langle \zeta, h_N\rangle \\
&\le C\|h_N\|^2+ \frac{\epsilon^2}{4} \|\Delta h_N\|^2 + C\|\zeta\|_{H^{-2}_{per}(\Omega)}^2 + \frac{\epsilon^2}{4} \|\Delta h_N\|^2,
\end{aligned}
$$
where $C$ is a positive general constant.
Combining the $\|\Delta h_N\|^2$ terms, multiplying the inequality by $2e^{-2Ct}$ and integrating against $t$, then using the fact that $e^{-CT}\le e^{-Ct}\le 1$ for $t\in[0,T]$,
we have for all $\tau\in [0,T_N]$
\begin{equation} \label{eq:hNglobal1}
\begin{aligned}
\|h_N(\cdot, \tau)\|^2 + &2\gamma  \|\nabla h_N\|^2_{L^2(0,\tau; L^2(\Omega))} + \epsilon^2 \|\Delta h_N\|^2_{L^2(0,\tau; L^2(\Omega))} \\
&\le \|h_N(\cdot,0)\|^2 + Ce^{2CT} \|\zeta\|^2_{L^2(0,\tau; H^{-2}_{per}(\Omega))} \le \|h_0\|^2 +  Ce^{2CT}\|\zeta\|^2_{L^2(0,\tau; H^{-2}_{per}(\Omega))},
\end{aligned}
\end{equation}
where the last step follows from Lemma \ref{lem:PNapproximation}. Thus one has
$$
\sum_{i=1}^{\Xi(N)} h_{N,i}^2 (\tau) = \|h_N(\cdot, \tau)\|^2 \le \|h_0\|^2+Ce^{2CT} \|\zeta\|^2_{L^2(0,\tau; H^{-2}_{per}(\Omega))} \quad\textrm{for all }\tau \in [0,T_N],
$$
i.e., the solution to System (\ref{eq:spectralODE}) is bounded at $T_N\le T$ as long as $h_0\in L^2(\Omega)$ and $\zeta\in L^2(0,T; H^{-2}_{per}(\Omega))$,
Hence a unique extension of the local solution to $[0,T]$, i.e. the global solution, exists.
Moreover, Inequality (\ref{eq:hNbound1}) follows immediately from (\ref{eq:hNglobal1}) and Lemma \ref{lem:PerLaplacian}.

To prove Inequality (\ref{eq:hNbound2}), we set $\phi=\partial_t h_N(t)$ in (\ref{eq:Spectraleq}) and use Lemma \ref{lem:gradHessianGk} to get
$$
\|\partial_t h_N\|^2 + \frac{d}{dt}\left[\frac{\gamma}{2}\|\nabla h_N\|^2 + \frac{\epsilon^2}{2}\|\Delta h_N\|^2 + \int_\Omega G(\nabla h_N)\, dx \right]
= \langle \zeta, \partial_t h_N\rangle \le \frac{1}{2}\|\zeta\|^2 + \frac{1}{2} \|\partial_t h_N\|^2.
$$
Combining the $\|\partial_t h_N\|^2$ terms and then integrating against $t$ give
\begin{equation} \label{eq:hNglobal2}
\begin{aligned}
\frac{1}{2}\|\partial_t h_N\|^2_{L^2(0,\tau; L^2(\Omega))} + &\frac{\gamma}{2}\|\nabla h_N(\cdot,\tau)\|^2 + \frac{\epsilon^2}{2}\|\Delta h_N(\cdot,\tau)\|^2
\le \frac{\gamma}{2}\|\nabla h_N(\cdot,0)\|^2 + \frac{\epsilon^2}{2}\|\Delta h_N(\cdot,0)\|^2 \\
 & + \left| \int_\Omega G(\nabla h_N(\cdot,\tau))\, dx \right| + \left| \int_\Omega G(\nabla h_N(\cdot,0))\, dx \right|
 + \frac{1}{2} \|\zeta\|^2_{L^2(0,\tau; L^2(\Omega))},
\end{aligned}
\end{equation}
for all $\tau\in [0,T]$.
Using Lemma \ref{lem:GkBounds1} and \ref{lem:PerLaplacian}, we have for all $\tau\in [0,T]$
\begin{equation} \label{eq:hNglobal3}
\begin{aligned}
\left| \int_\Omega G(\nabla h_N(\cdot,\tau))\, dx \right| &\le \int_\Omega |G(\nabla h_N(\cdot,\tau))|\, dx \le \int_\Omega (C + C |\nabla h_N(\cdot,\tau)| )\, dx \\
    &\le C + C \|\nabla h_N(\cdot,\tau)\|^2 \le C+C\|h_N(\cdot,\tau)\|^2 + \frac{\epsilon^2}{4}\|\Delta h_N(\cdot,\tau)\|^2,
\end{aligned}
\end{equation}
where $C$ is a positive general constant.
Combining (\ref{eq:hNglobal2})-(\ref{eq:hNglobal3}) and applying Lemma \ref{lem:PNapproximation}
as well as (\ref{eq:hNbound1}) give Inequality (\ref{eq:hNbound2}) except for the $\|h_N\|_{L^{2}(0,T; H^4(\Omega))}$ bound.

Finally, we shall estimate the $\|h_N\|_{L^{2}(0,T; H^4(\Omega))}$ bound in Inequality (\ref{eq:hNbound2}). Denote by $\partial$ any first-order
spatial derivative. Setting $\phi = -\partial^2 h_N$ in (\ref{eq:Spectraleq}) and using integration by parts to get
\begin{equation} \label{eq:hNH3}
\begin{aligned}
\frac{1}{2}\frac{d}{dt} \|\partial h_N\|^2 + \gamma \|\nabla\partial h_N\|^2 + \epsilon^2 \|\Delta \partial h_N\|^2
- \int_{\Omega} \partial \left[ \Phi(|\nabla h_N|) \nabla h_N\right] \cdot \nabla \partial h_N\, dx &= -\langle \zeta,\partial^2 h_N\rangle \\
&\le C\|\zeta\|^2 + \frac{\gamma}{4}\|\partial^2 h_N\|^2.
\end{aligned}
\end{equation}
By Lemma \ref{lem:PhiBounds}, especially noticing that $\Phi'(|\nabla h_N|)<0$, we have
$$
\partial \left[ \Phi(|\nabla h_N|) \nabla h_N\right] \cdot \nabla \partial h_N  =
 \Phi(|\nabla h_N|) |\nabla \partial h_N|^2 + \frac{\Phi'(|\nabla h_N|)}{|\nabla h_N|} (\nabla h_N\cdot \nabla \partial h_N)^2
\le C |\nabla \partial h_N|^2.
$$
Then, applying Lemma \ref{lem:PerLaplacian} to $\partial h_N$ and $h_N$ and using the Young's inequality give
$$
\begin{aligned}
\int_{\Omega} \partial \left[ \Phi(|\nabla h_N|) \nabla h_N\right] \cdot \nabla \partial h_N\, dx &\le C \|\nabla \partial h_N\|^2 \le C \|\partial h_N\|^2 \|\Delta \partial h_N\|^2 \\
&\le C\|\partial h_N\|^2 + \frac{\epsilon^2}{2} \|\Delta \partial h_N\|^2 .
\end{aligned}
$$
Substitute the above inequality into (\ref{eq:hNH3}), integrate against $t$, and use (\ref{eq:hNbound1}), one gets
\begin{equation} \label{eq:H3Bound}
\|h_N\|_{L^2(0,T;H^3(\Omega))} \le C( \|h_0\|_{H^2(\Omega)}, \|\zeta\|_{L^{2}(0,T; L^2(\Omega)}).
\end{equation}

Now similarly, set $\phi = \Delta^2 h_N$ in (\ref{eq:Spectraleq}), integrate against $t$ and apply the Young's inequality.
Again by Lemma \ref{lem:PhiBounds}, \ref{lem:PerLaplacian} and Sobolev embedding, we have
$$
\begin{aligned}
&\frac{\epsilon^2}{2} \int_0^T \|\Delta^2 h_N\|^2\, dt
\le  C( \|h_0\|^2_{H^2(\Omega)}, \|\zeta\|^2_{L^{2}(0,T; L^2(\Omega))})
    + \left|\int_0^T \int_{\Omega} \nabla\cdot(\Phi(|\nabla h_N|)\nabla h_N) \, \Delta^2 h_N \, dx\, dt\right| \\
= & C( \|h_0\|^2_{H^2(\Omega)}, \|\zeta\|^2_{L^{2}(0,T; L^2(\Omega))})
     + \left|\int_0^T \int_{\Omega} \bigg[\Phi'(|\nabla h_N|)\frac{(\nabla h_N)^t \nabla^2 h_N (\nabla h_N)}{|\nabla h_N|}
           + \Phi(|\nabla h_N|) \Delta h_N \bigg] \Delta^2 h_N \, dx\, dt\right| \\
 \le & C( \|h_0\|^2_{H^2(\Omega)}, \|\zeta\|^2_{L^{2}(0,T; L^2(\Omega))})
      + C \int_0^T \int_{\Omega}\bigg[ |\nabla h_N|\, |\nabla^2 h_N| + |\Delta h_N| \bigg] | \Delta^2 h_N| \, dx\, dt \\
 \le & C( \|h_0\|^2_{H^2(\Omega)}, \|\zeta\|^2_{L^{2}(0,T; L^2(\Omega))})
      + C \int_0^T \bigg[ \|h_N\|_{W^{1,4}(\Omega)}^2 \|h_N\|_{W^{2,4}(\Omega)}^2 + \|h_N\|_{H^2(\Omega)}^2 \bigg]  \, dt
       + \frac{\epsilon^2}{4} \int_0^T \|\Delta^2 h_N\|^2\, dt\\
\le & C( \|h_0\|^2_{H^2(\Omega)}, \|\zeta\|^2_{L^{2}(0,T; L^2(\Omega))},\|h_N\|_{L^2(0,T;H^3(\Omega))}^2) + \frac{\epsilon^2}{4} \int_0^T \|\Delta^2 h_N\|^2\, dt.
\end{aligned}
$$
This, together with Lemma \ref{lem:PerLaplacian} and Inequality (\ref{eq:H3Bound}), completes the proof of the lemma.
\end{proof}

\subsubsection{Existence and uniqueness of the weak solution}
Now we state the main existence, uniquess, and regularity theorem:

\begin{theorem} \label{thm:weakSolution}
Let $h_0\in H^2_{per}(\Omega)$ and $\zeta \in L^{2}(0,T; L^2(\Omega))$, then System (\ref{eq:diffeq}) has a unique weak solution
$h\in L^{\infty}(0,T; H^2_{per}(\Omega))\cap L^{2}(0,T; H^4_{per}(\Omega))$ with $\partial_t h\in L^2(0,T; L^2(\Omega))$.
\end{theorem}
\begin{proof}
We start from proving the existence of the weak solution.
Using Equation (\ref{eq:hNbound2}) and a compactness argument, it has been proved in Theorem 3.1 of \cite{LL03}
that there is a subsequence of Galerkin spectral approximations $h_N$ converging to a function $h\in L^{\infty}(0,T; H^2_{per}(\Omega))$  with $\partial_t h\in L^2(0,T; L^2(\Omega))$
in the following sense:
$$
h_N \stackrel{N\to \infty}{\xrightarrow{\hspace*{0.8cm}}} h \qquad \textrm{ in } L^{2}(0,T; H^1(\Omega)) \textrm{ strongly}.
$$
The rest of the proof of the existence also follows directly from the proof of Theorem 3.1 in \cite{LL03}.
Due to the different $\vm{J}_{ES}$ term, here we only need to re-prove Equation (3.14) in  \cite{LL03}.
That is to prove for any $\phi\in H^2_{per}(\Omega)$ and $\eta\in C[0,T]$,
\begin{equation} \label{eq:wellposednessthmproof1}
\left|\int_0^T \langle \eta(t)\nabla P_N\phi, \Phi(|\nabla h_N(\cdot,t)|)\nabla h_N(\cdot, t)\rangle\, dt
   -  \int_0^T \langle \eta(t)\nabla \phi, \Phi(|\nabla h(\cdot,t)|)\nabla h(\cdot, t)\rangle\, dt\right| \stackrel{N\to \infty}{\xrightarrow{\hspace*{0.8cm}}} 0.
\end{equation}
Indeed, by Lemma \ref{lem:PhiBounds}, \ref{lem:PNapproximation} and \ref{lem:GalerkinSolution}, the left-hand side of Equation (\ref{eq:wellposednessthmproof1}) satisfies
$$
\begin{aligned}
LHS &\le \left|\int_0^T \langle \eta(t)[\nabla P_N\phi-\nabla\phi], \Phi(|\nabla h_N(\cdot,t)|)\nabla h_N(\cdot, t)\rangle\, dt \right| \\
 &\qquad +\left|\int_0^T \langle \eta(t)\nabla \phi, [\Phi(|\nabla h_N(\cdot,t)|)-\Phi(|\nabla h(\cdot,t)|)]\nabla h(\cdot, t)\rangle\, dt\right| \\
 &\qquad\qquad + \left|\int_0^T \langle \eta(t)\nabla \phi, \Phi(|\nabla h_N(\cdot,t)|)[\nabla h_N(\cdot, t)-\nabla h(\cdot,t)]\rangle\, dt\right| \\
 &\le C \|\eta\|_{L^{\infty}(0,T)} \|P_N\phi-\phi\|_{H^1(\Omega)} \int_0^T \|\nabla h_N\|_{L^2(\Omega)}\, dt \\
 &\qquad + C\|\eta\|_{L^{\infty}(0,T)} \|\nabla \phi\|_{L^4(\Omega)} \|h_N-h\|_{L^2(0,T; H^1(\Omega))} \|\nabla h\|_{L^2(0,T;L^4(\Omega))} \\
 &\qquad\qquad + C\|\eta\|_{L^{\infty}(0,T)} \|\nabla \phi\|_{L^2(\Omega)} \int_0^T \|\nabla h_N-\nabla h\|\, dt \\
 & \le C(T)  \|P_N\phi-\phi\|_{H^1(\Omega)} + C(T) \|h_N-h\|_{L^2(0,T; H^1(\Omega))} \stackrel{N\to \infty}{\xrightarrow{\hspace*{0.8cm}}} 0,
\end{aligned}
$$
where in above we have used the fact that
$$
\left|\Phi(|\nabla h_N(\cdot,t)|)-\Phi(|\nabla h(\cdot,t)|)\right| \le \left(\sup_{s\ge0} |\Phi'(s)| \right) \bigg||\nabla h_N(\cdot,t)|-|\nabla h(\cdot,t)|\bigg|
\le C \bigg| \nabla h_N(\cdot,t) - \nabla h(\cdot,t) \bigg|.
$$
This completes the proof of the existence for the weak solution.

To prove the uniqueness of the solution, let $g$ and $h$ be two solutions of  System (\ref{eq:diffeq}) with initial conditions $g_0\in H^2_{per}(\Omega)$, $h_0\in H^2_{per}(\Omega)$
and right-hand side functions $\eta \in L^{2}(0,T; L^2(\Omega))$,  $\zeta \in L^{2}(0,T; L^2(\Omega))$, respectively. Define $w=g-h$.
According to the previous proof of existence, one has $w\in L^{\infty}(0,T; H^2_{per}(\Omega))$.
Clearly for all $\phi\in H^2_{per}(\Omega)$ and $t\in (0,T]$, $w$ satisfies
$$
\langle\partial_t w, \phi\rangle + \gamma \langle\nabla w,\nabla \phi\rangle + \epsilon \langle\Delta w,\Delta \phi\rangle
 - \langle\Phi (|\nabla g|) \nabla g-\Phi (|\nabla h|) \nabla h, \nabla\phi\rangle = \langle\eta-\zeta,\phi\rangle.
$$
Set $\phi=w(\cdot,t)$ in the above equation. By the mean value theorem,
Lemma \ref{lem:gradHessianGk}, Lemma \ref{lem:GkConcave}, Lemma \ref{lem:PerLaplacian} and the Young's inequality, one gets
$$
\begin{aligned}
\frac{1}{2}\frac{d}{dt}\|w\|^2 + \gamma\|\nabla w\|^2 + \epsilon^2\|\Delta w\|^2 
& = -\langle \nabla_F G(\nabla g) - \nabla_F G(\nabla h),\nabla w\rangle + \langle\eta-\zeta,w \rangle\\
&= -\langle \nabla_F^2 G(\vm{m}) \nabla w,\nabla w\rangle + \langle\eta-\zeta,w \rangle\\
&\le C \|\nabla w\|^2 + \|\eta-\zeta\|\, \|w\|\\
&\le \frac{\epsilon^2}{2}\|\Delta w\|^2 + C\|w\|^2 + \|\eta-\zeta\|^2,
\end{aligned}
$$
where $\vm{m}$ is a vector between $\nabla g$ and $\nabla h$ determined by the mean value theorem.
Then, by the Gr\"{o}nwall's inequality, we have
$$
\|w\|_{L^{\infty}(0,T;L^2(\Omega))}^2 + \|w\|_{L^{2}(0,T;H^2(\Omega))}^2 \le C(1+e^{CT})(\|g_0-h_0\|_{L^2(\Omega)}^2 + \|\eta-\zeta\|_{L^2(0,T;L^2(\Omega))}^2).
$$
The uniqueness of the weak solution follows immediately from the above estimate.

Finally, we prove the regularity of the solution, i.e., $h\in L^2(0,T; H^4_{per}(\Omega))$.
By Equation (\ref{eq:hNbound2}), $h_N$ is bounded in $L^2(0,T; H^4_{per}(\Omega))$.
Hence by passing to a subsequence if necessary, we have $h_N$ converges to $h$ weakly in $L^2(0,T; H^4_{per}(\Omega))$, i.e.,
$$
\|h\|_{L^2(0,T; H^4(\Omega))} \le \limsup_{N\to\infty} \|h_N\|_{L^2(0,T; H^4(\Omega))} \le C.
$$
This completes the proof of the theorem.
\end{proof}

\begin{rmk}
Similar to Theorem 3.3 in \cite{LL03}, one can achieve higher order regularity if $h_0$ and $\zeta$ are smoother. Also,
existence of weak solution can be proved with lower regularity requirement on $h_0$ and $\zeta$, if one uses different spaces
in the definition of the weak solution together with a refined compactness result, as discussed in  \cite{LL03}.
Here we skip these details in order to quickly get a functioning well-posedness result that allows us to
immediately start investigating numerical methods for the new model.
\end{rmk}

\subsection{Bounds of the solution and the roughness indicator}
As shown in the proof of Theorem \ref{thm:weakSolution},
by using the weak convergence of $h_N$ to $h$, one can pass the $\|h_N\|_{L^{\infty}(0,T,H^2(T))}\le C$ upper bound to $h$.
This upper bound, proved in Lemma \ref{lem:GalerkinSolution}, is just a rough estimate.
The purpose of this section is to derive more accurate upper bounds of the weak solution $h$ in certain norms or semi-norms.
To this end, we first state the following lemma:

\smallskip
\begin{lemma} \label{lem:PhikIntegral}
For all $\phi\in H^2_{per}(\Omega)$ that is mean value free, one has
\begin{equation}
\int_{\Omega} \Phi(|\nabla \phi|) |\nabla \phi|^2 \, dx \le \frac{\epsilon^2}{4}\|\Delta \phi\|^2 + C , \label{eq:PhikIntegral3}
\end{equation}
where $C$ is a positive general constant that does not depend on $\phi$.
\end{lemma}
\begin{proof}
When $\vm{J}_{ES}$ is set to $\vm{J}_{1,ES}$ defined in (\ref{J1}), by the definitions of $\Phi(\cdot)$, one has
$$
0\le \int_{\Omega} \Phi(|\nabla \phi|)|\nabla \phi|^2\, dx
=\alpha_1(q-p)\int_{\Omega}\frac{|\nabla \phi|^2}{(p+|\nabla \phi|)(q+|\nabla \phi|)}dx\le \alpha_1 (q-p)|\Omega|,
$$
and when $\vm{J}_{ES}$ is set to $\vm{J}_{2,ES}$ defined in (\ref{J3}), by Lemma \ref{lem:PerLaplacian} and Young's inequality, one has
$$
\begin{aligned}
0\le \int_{\Omega} \Phi(|\nabla \phi|)|\nabla \phi|^2\, dx &=\alpha_2\int_{\Omega}\frac{|\nabla \phi|^2}{q+ (p+q)|\nabla \phi|/p}dx
\le \frac{\alpha_2p}{p+q} \int_{\Omega} |\nabla \phi|\, dx \le \frac{\alpha_2p\sqrt{|\Omega|}}{p+q} \|\nabla \phi\| \\
&\le C\frac{\alpha_2 p\sqrt{|\Omega|}}{p+q}\|\Delta \phi\|
\le \frac{\epsilon^2}{4}\|\Delta \phi\|^2 + \frac{1}{\epsilon^2} \left(C\frac{\alpha_2 p\sqrt{|\Omega|}}{p+q}\right)^2.
\end{aligned}
$$
This completes the proof of the lemma.
\end{proof}

Now we can prove the following estimate:
\smallskip
\begin{lemma} \label{lem:hboundk1}
Let $h$ be the weak solution to (\ref{eq:diffeq}),
then one has for all $0\le t_0\le t\le T$,
$$
\frac{1}{2}\|h(\cdot, t)\|^2 + \gamma\int_{t_0}^t \|\nabla h\|^2\, d\tau + \frac{1}{2}\epsilon^2 \int_{t_0}^t \|\Delta h\|^2 \, d\tau
\le C(t-t_0) + \frac{1}{2}\|h(\cdot, t_0)\|^2 + \frac{1}{\epsilon^2} \|\zeta\|^2_{L^2(t_0,t;H^{-2}_{per}(\Omega))},
$$
where $C$ is the same constant as defined in Lemma \ref{lem:PhikIntegral},
which depends only on $\alpha_k$, $p$, $q$, $|\Omega|$, $\epsilon$ and $\gamma$.
\end{lemma}
\begin{proof}
Setting $\phi=h$ in (\ref{eq:weakeq}) and using Lemma \ref{lem:PhikIntegral} give
$$
\begin{aligned}
\frac{1}{2}\frac{d}{dt}\|h\|^2 + \gamma\|\nabla h\|^2 + \epsilon^2\|\Delta h\|^2 &= \int_{\Omega}\Phi(|\nabla h|)|\nabla h|^2\, dx + \langle \zeta,h\rangle \\
&\le C + \frac{\epsilon^2}{4}\|\Delta h\|^2 + \frac{1}{\epsilon^2}\|\zeta\|_{H^{-2}_{per}(\Omega)} + \frac{\epsilon^2}{4}\|\Delta h\|^2.
\end{aligned}
$$
Combine common terms and then integrate from $t_0$ to $t$, this completes the proof of the lemma.
\end{proof}
\smallskip
\begin{rmk}
Lemma \ref{lem:hboundk1} states that $\|h\|^2$, $\int_{t_0}^t \|\nabla h\|^2\, d\tau$, and $\int_{t_0}^t \|\Delta h\|^2 \, d\tau$
have at most linear growth rate with respect to $t$, starting from any point $0\le t_0\le T$.
This is globally better than exponential growth.
\end{rmk}
\smallskip

\begin{rmk}
In \cite{LL04}, the authors have studied the evolution of the surface roughness indicator for the finite ES barrier with slop selection case, i.e.,
the ES currents is $\vm{J}_{FSS,ES}$. Here, by Lemma \ref{lem:hboundk1}, we immediately have a same surface roughness evolution bound
for the ES current $\vm{J}_{k,ES}$ with $k=1,2$, since by definition one has $|\Omega|\omega^2 = \|h\|^2$ for $t\in[0,T]$.
When $\zeta =0$, one gets a global bound for the growth of $\omega(t)$:
\begin{equation} \label{eq:omegaGlobal}
\omega(t)\le \sqrt{C(t-t_0) + \omega^2(t_0)}\qquad \textrm{for all }0\le t_0\le t \le T.
\end{equation}
\end{rmk}

Next, we aim at deriving a local bound for the growth of $\omega(t)$ that is better than the global bound when the value of $\omega(t)$ is small.
To this end, we first point out that the upper bound in Lemma \ref{lem:PhikIntegral} is not very sharp when $|\nabla \phi |$ is small.
For small $|\nabla \phi |$, one shall consider the following alternative bound:
\smallskip

\begin{lemma} \label{lem:PhikIntegralAlt}
For all $\phi\in H^2_{per}(\Omega)$ that is mean value free, one has
\begin{equation}
\int_{\Omega} \Phi(|\nabla \phi|) |\nabla \phi|^2 \, dx \le \frac{\epsilon^2}{4}\|\Delta \phi\|^2 + C\|\phi\|^2 , \label{eq:PhikIntegralAlt}
\end{equation}
where $C$ is a positive general constant that does not depend on $\phi$.
\end{lemma}
\begin{proof}
By lemmas \ref{lem:PhiBounds}, \ref{lem:PerLaplacian} and the Young's inequality, one immediately has
$$
\int_{\Omega} \Phi(|\nabla \phi|) |\nabla \phi|^2 \, dx \le C\|\nabla \phi\|^2 \le \frac{\epsilon^2}{4}\|\Delta \phi\|^2 + C\|\phi\|^2.
$$
\end{proof}
\begin{rmk}
Comparing to Lemma \ref{lem:PhikIntegral}, Lemma \ref{lem:PhikIntegralAlt} is better when $\|\nabla \phi\|$ is small
(or consequently, by the Poincar\'{e} inequality, when $\|\phi\|$ is small).
\end{rmk}

Now we can derive the following local bound for $\omega(t)$:

\begin{lemma} \label{lem:localBoundOmega}
Let $h$ be the weak solution to (\ref{eq:diffeq}) and assume $\zeta = 0$. Then one has for all $0\le t_0\le t \le T$,
\begin{equation} \label{eq:omegaLocal}
\omega(t) \le \sqrt{\omega^2(t_0) e^{C(t-t_0)}}.
\end{equation}
\end{lemma}
\begin{proof}
Setting $\phi=h$ in (\ref{eq:weakeq}) and using Lemma \ref{lem:PhikIntegralAlt} give
$$
\frac{1}{2}\frac{d}{dt}\|h\|^2 + \gamma\|\nabla h\|^2 + \epsilon^2\|\Delta h\|^2 = \int_{\Omega}\Phi(|\nabla h|)|\nabla h|^2\, dx
\le C\|h\|^2 + \frac{\epsilon^2}{4}\|\Delta h\|^2 ,$$
which implies
$$
 \frac{d}{dt}\|h\|^2  \le C\|h\|^2.
$$
Solve the ordinary differential equation and use $|\Omega|\omega^2 = \|h\|^2$, one gets the local bound (\ref{eq:omegaLocal}).
\end{proof}

\begin{figure}[h]
\begin{center}
\includegraphics[width=5cm]{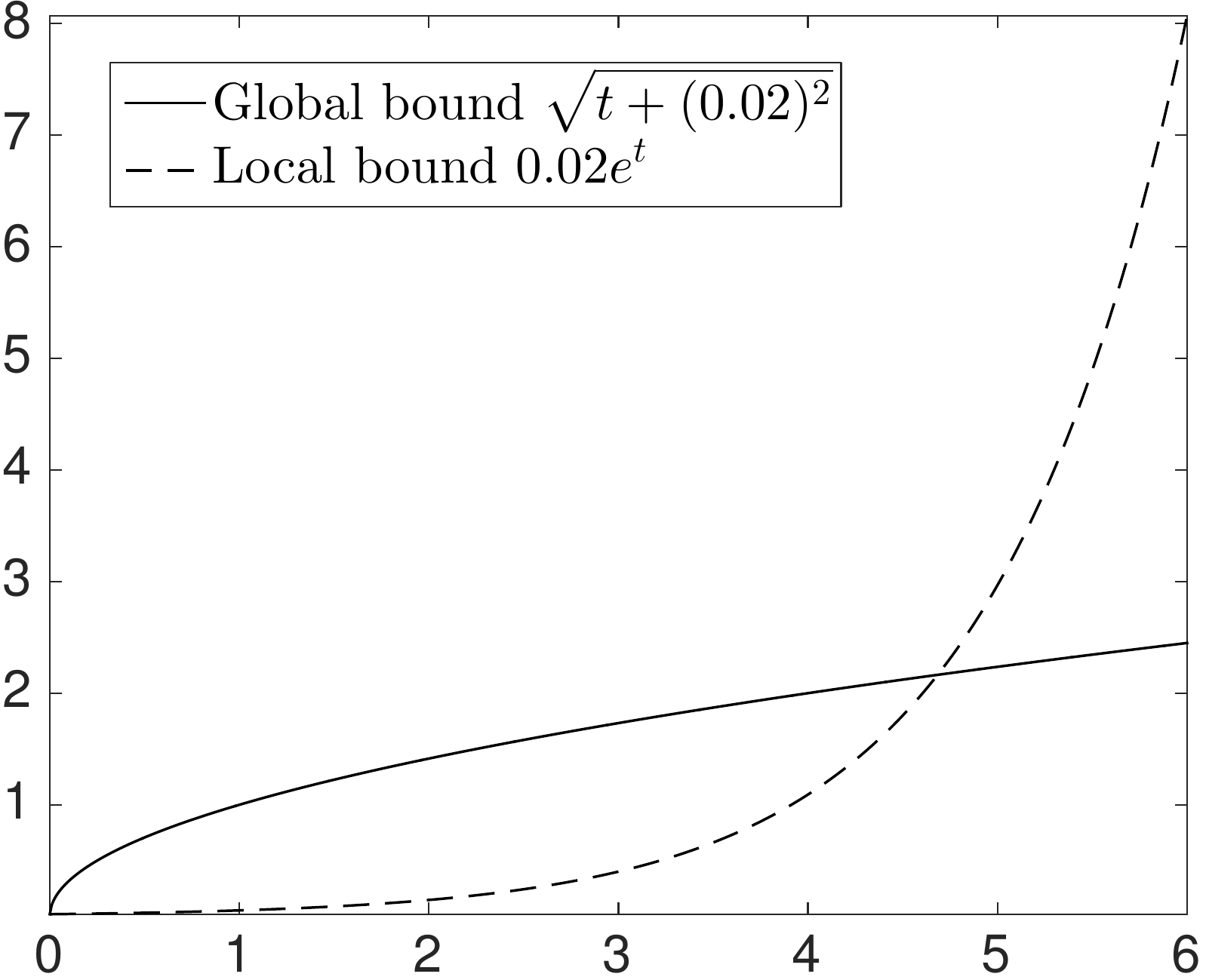}
\caption{Difference between the global bound (\ref{eq:omegaGlobal}) and the local bound (\ref{eq:omegaLocal}), illustrated
in an example with $t_0=0$, $\omega(0) = 0.02$ and $C = 1$. The local bound is better when $t$ is small.
More importantly, the local bound prescribes a ``flat" start-up process in the beginning of the evolution, when $\omega(t_0)$ is small.} \label{fig:globalAndLocal}
\end{center}
\end{figure}

\begin{rmk}
Generally speaking, when $t-t_0$ is large, the global bound (\ref{eq:omegaGlobal}), with growth rate $O(\sqrt{t-t_0})$, is
much smaller than the local bound (\ref{eq:omegaLocal}), which is an exponential growth.
However, when $\omega(t_0) \approx 0$ and $t-t_0$ is small, the local bound (\ref{eq:omegaLocal}) becomes
 smaller than the global bound (\ref{eq:omegaGlobal}).
An illustration is given in Figure \ref{fig:globalAndLocal}.
The main reason that we derive the local bound is to explain that the local growth of $\omega(t)$, when the value of $\omega(t)$ is small,
is indeed ``flat" rather than the ``abrupt" growth prescribed by (\ref{eq:omegaGlobal}).
Later, this pattern can be observed in numerical results.
Other intermediate estimates between the local bound and the global bound can also be obtained easily through interpolation. 
For example, one can prove that $\omega(t)\le (C(t-t_0)+w^\frac{4}{3}(t_0))^{\frac{3}{4}}$. 
Such intermediate estimates may be more accurate to describe the growth rate in some stages of the evolution.
Just as we know, to strictly depict the growth rate in different stages is still one difficult task.
\end{rmk}

\begin{rmk}
When the ES current is taken to be $\vm{J}_{3,ES}$, we do not have an existence and uniqueness theory due to the unboundedness of
$\vm{J}_{3,ES}$ at $|\nabla h|=0$. However, if there exists a weak solution, then the global and local bounds in
(\ref{eq:omegaGlobal}) and (\ref{eq:omegaLocal}) will apply, because one can easily show that for all $\phi\in H^2_{per}(\Omega)$,
$$
\begin{aligned}
\int_{\Omega} \Phi_3(|\nabla \phi|) |\nabla \phi|^2 \, dx &= \int_{\Omega} \alpha_3\frac{|\nabla \phi|}{|\nabla \phi|+(q-p)}\, dx
   \le \alpha_3 |\Omega| = C,  \\
\int_{\Omega} \Phi_3(|\nabla \phi|) |\nabla \phi|^2 \, dx  &= \int_{\Omega} \alpha_3\frac{|\nabla \phi|}{|\nabla \phi|+(q-p)}\, dx
   \le \frac{\alpha_3}{q-p} \int_{\Omega} |\nabla \phi|\, dx
   \le \frac{\epsilon^2}{4}\|\Delta \phi\|^2 + C\|\phi\|^2 ,
\end{aligned}
$$
which are parallel to the results in lemmas \ref{lem:PhikIntegral} and \ref{lem:PhikIntegralAlt}.
Thus inequalities (\ref{eq:omegaGlobal}) and (\ref{eq:omegaLocal}) hold.
\end{rmk}
\subsection{Energy functional} \label{sec:energyfunctional}
Define an energy functional associated with Equation (\ref{eq:diffeq}) as follows
$$
E(h) = \int_\Omega \left[ G(\nabla h) + \frac{\gamma}{2} |\nabla h|^2 + \frac{\epsilon^2}{2} |\Delta h|^2 \right] \, dx.
$$
Then the differential equation (\ref{eq:diffeq}) can be written as $\partial_t h = \zeta -\frac{\delta E(h)}{\delta h}$,
where
$\frac{ \delta E (h)}{\delta h}$ denotes the Fr\'{e}chet derivative of the energy functional with respect to $h$.

Note that by definition $G(\nabla h)$ can be negative, which implies that the energy functional $E(h)$ can be negative too.
However, by using the Poincar\'{e} inequality, Lemma \ref{lem:PerLaplacian} and choosing $\beta$ carefully, one immediately has the following lower bound of $E(h)$:
\smallskip
\begin{lemma} \label{lem:energyLowerBound}
The energy functional has lower bound
$$
E(h) \ge -C,
$$
where $C$ is a positive constant independent of $h$.
\end{lemma}
\smallskip

Another important observation is:
\smallskip
\begin{lemma} \label{lem:energyStability}
Let $\zeta\in L^{\infty}(0,T,L^2(\Omega))$, the energy functional $E(h)$ satisfies
$$
\frac{d}{dt} E(h) \le \frac{1}{2}\|\zeta\|^2 - \frac{1}{2} \|\partial_t h\|^2, \qquad\textrm{for }0\le t\le T.
$$
Note that $E(h)$ is non-increasing when $\zeta=0$.
\end{lemma}
\begin{proof}
Setting $\phi = \partial_t h$ in (\ref{eq:weakeq}) gives
$$
\|\partial_t h\|^2 + \frac{d}{dt} E(h) = \langle \zeta, \partial_t h\rangle \le \|\zeta\|\,\|\partial_t h\| \le \frac{1}{2}\|\zeta\|^2 + \frac{1}{2}\|\partial_t h\|^2.
$$
This completes the proof of the lemma.
\end{proof}
\smallskip

We shall discuss a little more about the special case when $\zeta=0$. In this case,
it is not hard to see that there is a one-to-one correspondence between the critical points of $E(\cdot)$,
i.e., $h$ satisfying $\frac{\delta}{\delta h} E(h) \triangleq \frac{d}{d\tau} E(h+\tau\phi)|_{\tau=0} = 0$,
and steady state solutions of (\ref{eq:diffeq}). Moreover, local minimums of $E(\cdot)$ give stable steady state solutions of (\ref{eq:diffeq}).
We also point out that $h\equiv 0$ is obviously a steady state solution in this case.
By Lemma \ref{lem:GkBounds1}, $E(h)$ becomes a convex functional of $\nabla h$ when $\gamma$ is large enough.
In this case, there exists a unique global minimizer, which is $h\equiv 0$.
Note here $h\equiv C$ is not considered because of the mean value free assumption.

\begin{rmk}
It is worth to point out that when the ES current is taken to be $\vm{J}_{3,ES}$, the energy functional is still well defined
since $G_3 \in C(\bbR^2)$ although it is not differentiable at the origin. Moreover, the lower bound in Lemma \ref{lem:energyLowerBound}
is also true because
$$
G_3(\nabla h) = -\alpha_3 \ln(q-p+|\nabla h|) \ge -C |\nabla h|-C \ge -\beta |\nabla h|^2 -C_{\beta},
$$
for any positive constant $\beta$.
\end{rmk}

\section{A semi-implicit fully-discrete numerical scheme}\label{Sprop}
In this section we develop a semi-implicit numerical scheme for approximating
the weak solution of (\ref{eq:diffeq}), using the technique of convex-splitting.
For illustrative purpose, the spatial discretization uses the Galerkin spectral
approximation with discrete space $H_N$ presented in Section \ref{Smodel}, though we point out that the scheme and
analysis also apply to other Galerkin approximations.
The numerical scheme is stated below.
We first split the form $a(\cdot,\cdot)$ defined in (\ref{eq:weakeq}) into two parts:
$$
\begin{aligned}
a_+(h,\phi) &= \gamma \langle\nabla h,\nabla \phi\rangle + \epsilon^2 \langle\Delta h,\Delta \phi\rangle + \langle \nabla_F G_+(\nabla h), \nabla\phi\rangle, \\
a_-(h,\phi) &= \langle \nabla_F G_-(\nabla h), \nabla\phi\rangle,
\end{aligned}
$$
which satisfy $a(h,\phi) = a_+(h,\phi) + a_-(h,\phi)$.

Given $\delta t = \frac{T}{M}$, where $M$ is a positive integer, and define $t_i = i\,\delta t$, for $i=0,\ldots,M$.
Denote by $h_N^i\in H_N$ and $\zeta^i=\zeta(\cdot, t_i)$, for $i=0,\ldots,M$ the numerical approximation and the deposition rate, respectively, at $t_i$.
Then the semi-implicit time discretization can be written as:
\begin{enumerate}
\item Set $h_N^0 = P_N h_0$;
\item For $i=1,\ldots, M$, compute $h_N^i$ by
\begin{equation} \label{eq:scheme}
\langle\frac{h_N^i-h_N^{i-1}}{\delta t}, \phi\rangle + a_+(h_N^i, \phi)
= \langle\zeta^i,\phi\rangle - a_-(h_N^{i-1},\phi),
\end{equation}
for all $\phi\in H_N$.
\end{enumerate}

By the definition of $G_+$ and $G_-$ in Corollary \ref{cor:convexsplitting}, it is
clear that $a_+(\cdot,\cdot)$ is a symmetric and coercive bilinear form, while $a_-(\cdot,\cdot)$ is nonlinear.
Hence the scheme is uniquely solvable at each time step.
Next, we consider the energy stability of the scheme.

For simplicity, denote
$$
\mathcal{G}_{+}(\phi) = \int_{\Omega} G_{+}(\nabla \phi) \, dx,\qquad \mathcal{G}_{-}(\phi) = \int_{\Omega} G_{-}(\nabla \phi) \, dx,
$$
for $\phi\in H^2_{per}(\Omega)$ and $k=1,2$.

\smallskip
\begin{lemma} \label{lem:convexconcave}
For $\phi,\psi\in H^2_{per}(\Omega)$, we have
$$
\begin{aligned}
\mathcal{G}_{+}(\phi) - \mathcal{G}_{+}(\psi) &\le \int_{\Omega} \nabla_F G_{+}(\nabla \phi)\cdot \nabla (\phi-\psi)\, dx, \\
\mathcal{G}_{-}(\phi) - \mathcal{G}_{-}(\psi) &\le \int_{\Omega} \nabla_F G_{-}(\nabla \psi)\cdot \nabla (\phi-\psi)\, dx.
\end{aligned}
$$
\end{lemma}
\begin{proof}
By the mean value theorem and the fact that $G_{+}$ is convex, one has
$$
\begin{aligned}
\mathcal{G}_{+}(\phi) - &\mathcal{G}_{+}(\psi) - \int_{\Omega} G_{+}(\nabla \phi) \cdot \nabla (\phi-\psi)\, dx \\
&= \int_{\Omega} \bigg[ G_{+}(\nabla \phi)- G_{+}(\nabla \psi) -  \nabla_F G_{+}(\nabla \phi) \cdot \nabla (\phi-\psi) \bigg]\, dx \\
&= \int_{\Omega} \bigg[ \nabla_F G_{+}(\nabla \phi-s_1\nabla (\phi-\psi)) -  \nabla_F G_{+}(\nabla \phi)  \bigg]\cdot \nabla(\phi - \psi)\, dx \\
&= \int_{\Omega} \bigg[ \nabla_F^2 G_{+}(\nabla \phi-s_2\nabla (\phi-\psi)) (-s_1\nabla(\phi-\psi)) \bigg]\cdot \nabla(\phi - \psi)\, dx \\
&\le 0,
\end{aligned}
$$
where $0\le s_1\le s_2\le 1$ are constants determined by the mean value theorem.
The proof for $G_{-}$ is similar.
\end{proof}
\smallskip

Define the discrete energy at each time step $i=0,\cdots, M$ by
$$
E^i = E(h_N^i) = \frac{\gamma}{2}\|\nabla h_N^i\|^2 + \frac{\epsilon^2}{2}\|\Delta h_N^i\|^2 + \int_{\Omega} G(\nabla h_N^i)\, dx.
$$
By Lemma \ref{lem:energyLowerBound}, we know that $E^i \ge -C$ where $C>0$ is a constant independent of $h_N^i$.
More over, we have the following energy stability:
\smallskip
\begin{lemma}
The scheme (\ref{eq:scheme}) is unconditionally energy stable in the sense of
$$
E^i \le E^{i-1} + \frac{\delta t}{2} \|\zeta^i\|^2 - \frac{\|h_N^i-h_N^{i-1}\|^2}{2\delta t} - \frac{\gamma}{2}\|\nabla(h_N^i-h_N^{i-1})\|^2
 - \frac{\epsilon^2}{2}\|\Delta(h_N^i-h_N^{i-1})\|^2,
$$
for all $i=1,\cdots, M$. Note that when $\zeta^i\equiv 0$, one has $E^i\le E^{i-1}$.
\end{lemma}
\begin{proof}
Set $\phi = h_N^i-h_N^{i-1}$ in (\ref{eq:scheme}) and use Lemma \ref{lem:convexconcave} as well as the fact $a(a-b) = \frac{1}{2}(a^2-b^2+(a-b)^2)$, we have
$$
\begin{aligned}
\frac{\|h_N^i-h_N^{i-1}\|^2}{\delta t} + E^i-E^{i-1} + \frac{\gamma}{2}\|\nabla(h_N^i-h_N^{i-1})\|^2 + \frac{\epsilon^2}{2}\|\Delta(h_N^i-h_N^{i-1})\|^2
&\le \langle \zeta^i, h_N^i-h_N^{i-1}\rangle \\
\le \frac{\delta t}{2} \|\zeta^i\|^2 + \frac{\|h_N^i-h_N^{i-1}\|^2}{2 \delta t}.
\end{aligned}
$$
This completes the proof of the lemma.
\end{proof}
\smallskip
\begin{rmk}
The above energy stability indeed also implies $H^2$ stability. By lemmas \ref{lem:PerLaplacian} and \ref{lem:GkBounds2}, one has
$$
E^i \ge \frac{\epsilon^2}{2}\|\Delta h_N^i\|^2 - \beta \|\nabla h_N^i\|^2 - C_{\beta}|\Omega|
\ge \frac{\epsilon^2}{2}\|\Delta h_N^i\|^2  - C_P\beta \|\Delta h_N^i\|^2  - C_{\beta}|\Omega| ,
$$
where $\beta$ can be any positive constant and $C_P$ is a constant from the Poincar\'{e} inequality.
Choose $\beta$ such that $C_P\beta = \frac{\epsilon^2}{4}$, one has
$$
E^i+ C_{\beta}|\Omega| \ge \frac{\epsilon^2}{4}\|\Delta h_N^i\|^2 \ge C\frac{\epsilon^2}{4} \|h_N^i\|^2_{H^2(\Omega)}.
$$
Thus the scheme is also $H^2$ stable.
Note here we avoid using $\frac{\gamma}{2}\|\nabla h_N^i\|^2$ to control $\beta \|\nabla h_N^i\|^2$, because $\gamma$ can be $0$.
\end{rmk}

Finally, we study the error estimate of the numerical scheme. Denote the error at each time step $t_i$, for $i=0,\ldots,M$ by
$$
\underline{h}^i = h(\vm{x}, t_i) - h_N^i (\vm{x}),
$$
where $h$ is the weak solution to (\ref{eq:diffeq}) and $h_N^i$ is the numerical solution.
The proof of the following theorem is quite standard and we thus postpone it to Appendix \ref{appendix}.

\smallskip
\begin{theorem} \label{th3}(Error estimate).
Let $h_0\in H^2_{per}(\Omega)$ and $\zeta \in L^{2}(0,T; L^2(\Omega))$.
Assume the weak solution to (\ref{eq:diffeq}) satisfy $h_{tt}\in L^2(0, t_n; L^2(\Omega))$,
$h_t\in L^2(0, t_n; H^{m_1}(\Omega))$ with $m_1\ge 2$, and $h\in L^{\infty}(0, t_{n-1}; H^{m_2}(\Omega))$ with $m_2 \ge 2$.
Then
$$
\|\underline{h}^i\| \le Ce^{Ct_n} (\delta t + N^{-m_1} + N^{-(m_2-1)}).
$$
\end{theorem}

\begin{rmk}
Although we used the Fourier spectral Galerkin method in the spatial discretization,
similar results hold for other Galerkin spatial discretizations.
\end{rmk}

%

\section{Numerical Results}\label{Snumer}
An important feature of the semi-implicit fully discrete scheme (\ref{eq:scheme}) is that,
it has a linear implicit part and hence one only needs to solve a linear problem in every time step.
This is a great advantage comparing to the numerical schemes for models using $\vm{J}_{F,ES}$ or $\vm{J}_{FSS,ES}$
as the ES current, which inevitably require a nonlinear implicit part for stability purpose \cite{LL03, CW}.
However, one may wonder whether the new model, although easier to compute, can still correctly capture the evolution of surface morphology or not.
In this section, we will first answer this question by comparing the numerical results from the new model with numerical results
from other models reported in \cite{LL03}. We will also test the new model on a larger set of examples to examine its performance.

Set $\Omega=(0,2\pi)^2$ and choose the Fourier spectral Galerkin approximation as the spatial discretization
in Scheme (\ref{eq:scheme}). We pick this spatial discretization because it can be easily and efficiently implemented in Matlab using
the build-in Fast Fourier Transform (FFT) tool. In all numerical experiments, set the size of spatial discretization to be $N^2$ with $N=128$, i.e.,
in the physical space $h$ is evaluated on a $128\times 128$ grid and consequently in the frequency space $\hat{h}$ is
approximated by $128\times 128$ Fourier modes, where $\hat{h}$ denoted the discrete Fourier transform of $h$.
For simplicity, assume $\zeta=0$.

Although $\vm{J}_{3,ES}$ is not continuous at $\nabla h = \vm{0}$, we have shown in Remark \ref{rem:J3}
that it still possesses several nice properties including the most important convex-concave splitting property.
Thus we are also interested in testing $k=3$ numerically and comparing it with $k=1,2$.
To distinguish between different ES currents $\vm{J}_{1,ES}$, $\vm{J}_{2,ES}$ and $\vm{J}_{3,ES}$,
here again we shall adopt the subscript $k=1, 2, 3$ throughout the rest of this section.
In the implementation, one has to deal with the calculation of $\vm{J}_{3,ES}$ when $|\nabla h| = 0$.
Here we adopt a makeshift solution by setting $|\nabla h| = \max\{|\nabla h|,\, 10^{-16}\}$.

Next we shall consider proper choice of the parameters, such as $\alpha$, $p$, $q$, $\gamma$ and $\epsilon$, in Equation (\ref{eq:diffeq}).
To make a realistic choice, let us first recall how the model was built from physical laws.
According to \cite{Evans}, set
$$
\begin{aligned}
\alpha_1&=\frac{F_fL_{ES}}{2}, \quad &\alpha_2&=\alpha_1\frac{q-p}{p},\quad &\alpha_3&=\alpha_1(q-p),
\quad &p&=\frac{b}{L_{isl}}, \quad &q&=\frac{b}{L_{isl}}+\frac{b}{L_{ES}},\\
\gamma&=C_{DF}F_f,\quad & \epsilon^2&=F_f(L_{isl})^4,
\end{aligned}
$$
where $F_f$ is the deposition flux per unit time, $L_{ES}$ is the adatom attachment length when descending a step (ES effect),
$L_{isl}$ is the typical island separation length, $b$ is the typical step height,
and $C_{DF}$ is the strength of the downward funneling current.
We start from setting  $L_{isl}=0.25$, $L_{ES}=0.05$, $b=0.017$, $C_{DF}=0$ and $F_{f}=2$, which gives
$$
\alpha_1 = 0.05,\quad \alpha_2 = 0.25,\quad \alpha_3 = 0.017,\quad p = 0.068,\quad q = 0.408, \quad \gamma = 0,\quad \epsilon^2 = 0.0078.
$$
Later we shall perturb parameters $\gamma$ and $\epsilon^2$ a little bit to investigate their effect on the surface evolution.
Note that Equation (\ref{eq:diffeq}) is linear in terms of $t$, therefore scaling $\alpha$, $\gamma$ and $\epsilon^2$ together
is equivalent to changing the time scale. Thus we do not plan to test the numerical scheme for different values of $\alpha$.
Also, because of the small value of $\alpha$ we currently pick, the surface evolution with respect to time appear to be relatively slow.
Hence we have found that setting the time step size $\delta t= 0.01$ is adequate to resolve the rich details of the evolution.
Though we point out that one may choose any other time step size and the numerical scheme will always be stable as proved in Section \ref{Sprop}.
However, $\delta t$ should be small enough in order to attain certain accuracy.
One may also consider adaptive time-stepping strategies such as the one proposed in \cite{QZ11}.

\subsection{Example 1}
We start from the initial condition used in \cite{LL03}:
$$
h_0 = 0.1(\sin 3x\sin 2y + \sin 5x \sin 5y).
$$
Surface evolution with this initial condition using other models have been numerically studied in details in \cite{LL03}.
Interestingly, our numerical results show that the new model produces highly similar evolution patterns as those reported in  \cite{LL03},
despite the different ES current $\vm{J}_{ES}$ used in these models. Note all these models are constructed based on the same physical phenomena,
thus the numerical similarity indicates that they have each individually models the microscopic movement relatively correct.

\begin{figure}[ht]
\begin{center}
\includegraphics[width=10cm]{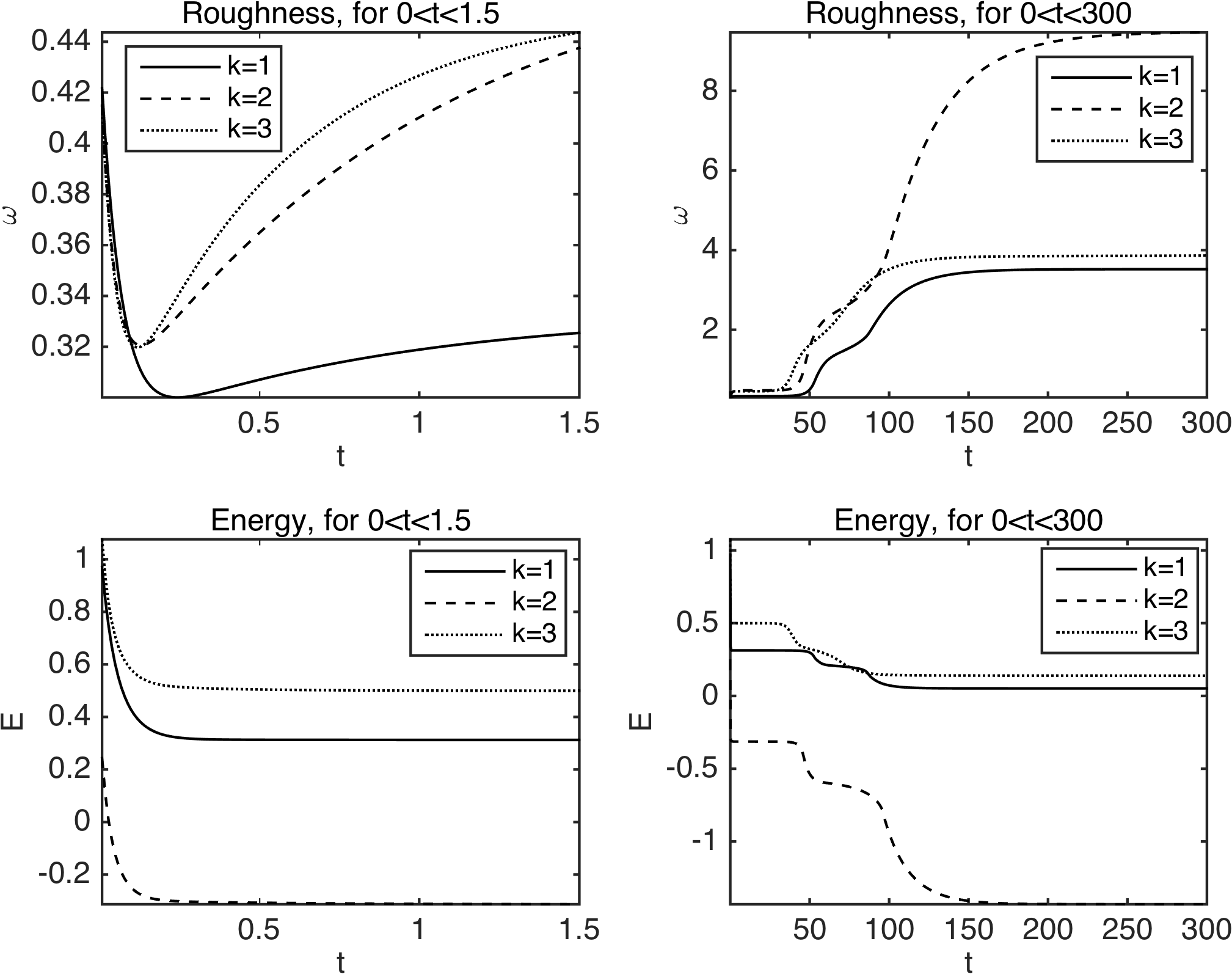}
\caption{Evolution of roughness $\omega$ and energy $E$ for Example 1. Left: early evolution; Right: entire evolution.} \label{fig:E1-roughnessenerygy}
\end{center}
\end{figure}

\begin{figure}[ht]
\begin{center}
\includegraphics[width=14cm]{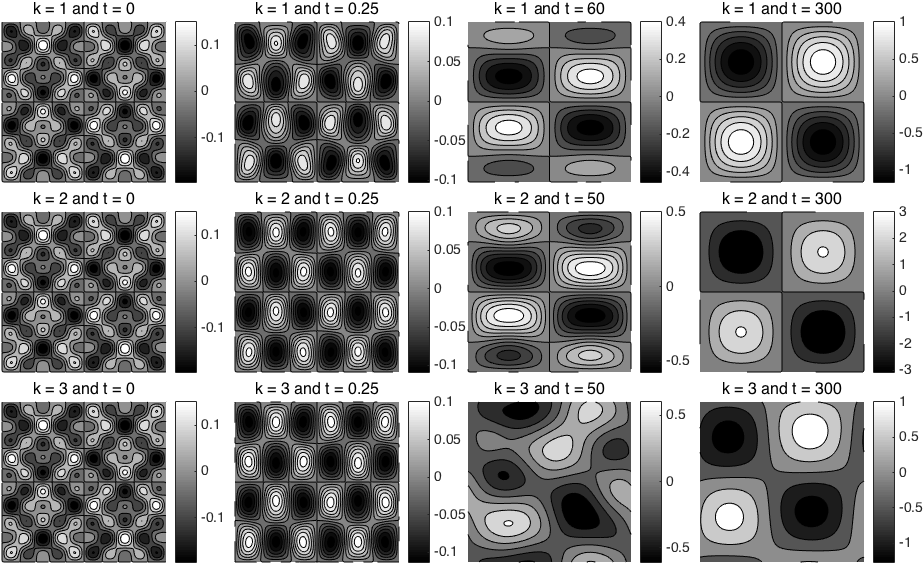}
\caption{Evolution of surface pattern for Example 1.}  \label{fig:E1-surfacek1}
\end{center}
\end{figure}

\begin{figure}[ht]
\begin{center}
\includegraphics[width=10cm]{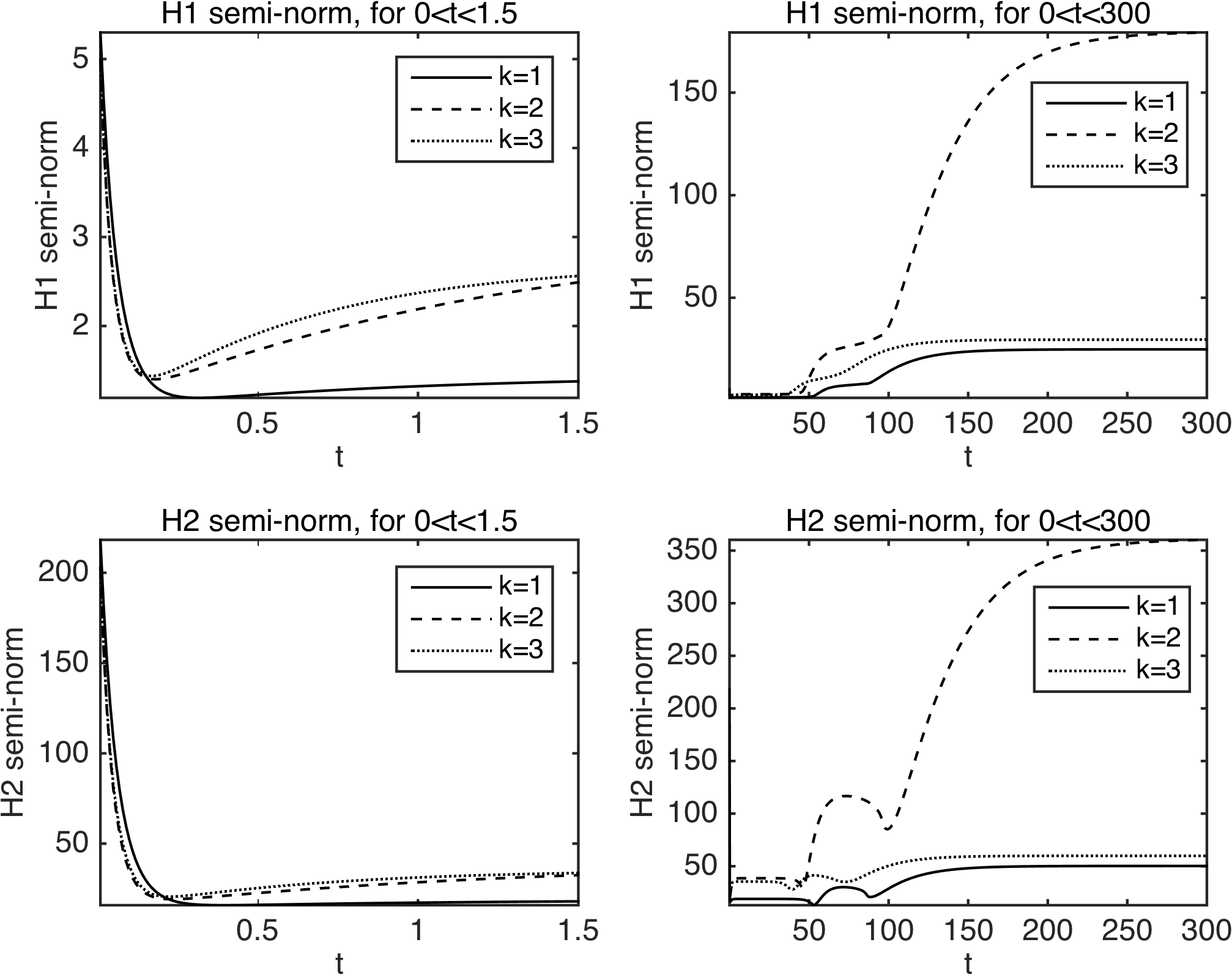}
\caption{Evolution of $\|\nabla h\|^2$ and $\|\Delta h\|^2$ for Example 1. Left: early evolution; Right: entire evolution.} \label{fig:E1-H1H2}
\end{center}
\end{figure}

The evolution of surface roughness $\omega$ and energy $E$ is reported in Figure \ref{fig:E1-roughnessenerygy}.
It seems that a steady state solution has been reached at the end, as the roughness and energy curves appear to be flat.
We shall point out that although the graph only shows evolution for $0\le t\le 300$,
the actual computation is done for a much longer time period, in order to ensure that the roughness and energy curves stay flat at the end.
The same holds for all numerical results reported in this section.
Contour plots of the solutions at different time steps in Figure \ref{fig:E1-surfacek1} suggest
that the case $k=3$ suffers from numerical round-off errors probably introduced by the crude treatment of $|\nabla h|=0$.

In all cases, notice that the energy in Figure \ref{fig:E1-roughnessenerygy} is non-increasing, i.e., the scheme is energy stable. Comparing with
reports in \cite{LL03}, they also share the following similarities:
\begin{enumerate}
\item The roughness $\omega$ drops in the beginning and then starts to increase, which has been described as a ``rough-smooth-rough" pattern in \cite{LL03}.
\item The evolution goes through several ``flat" stages, with each ``flat" stage corresponding to a relatively stable surface pattern in the coarsening process.
      Similar ``flat" stages have been reported in \cite{LL03}.
\item We draw the surface image at different time, and report them in Figure \ref{fig:E1-surfacek1}. Note that
    for either $k=1$ or $k=2$, the coarsening process evolves through three very different patterns, including the final steady state solution.
    We point out that these three patterns have exactly the same structure as the stages reported in \cite{LL03}.
    Also, there is no structural difference between the surface evolution for $k=1$ and $k=2$.
    The case $k=3$ is slightly different than $k=1,2$ as the solution obviously is smeared by artificial round-off error in the middle of the evolution,
    which we suspect is introduced through the crude treatment of $|\nabla h|=0$.
    Recall that the theoretical well-posedness and error estimate do not work for $k=3$.
    However, it is interesting to see that that numerical scheme for $k=3$ remains energy stable, and
    its solution eventually converges to a steady state of the same pattern as the solutions for $k=1,2$,
    with a small phase shift.
    In examples to be given later, we will see that there exist multiple types of steady state solutions for our model problem.
    But for a given initial condition, the behavior of solutions for $k=1,2,3$ seems to be similar and all three converge to the same type of steady state solution.
\item We also point out that our numerical results have a longer evolution time length comparing to results in  \cite{LL03}.
    This is because we have picked small values for $\alpha$ and $\epsilon^2$, which result in a slower coarsening process.
\end{enumerate}

The squares of semi-norms, i.e. $\|\nabla h\|^2$ and $\|\Delta h\|^2$, of the solution for Example 1 are reported in Figure \ref{fig:E1-H1H2}.
The reason why we report squares of semi-norms instead of the semi-norms is that the squares instead of the semi-norms are components of $E(h)$.
One immediately notice that
 $\|\nabla h\|^2$ grows in almost the same pattern as $\omega$, while $\|\Delta h\|^2$ appears to have sudden drops at the transition between ``flat" stages.

\begin{figure}[ht]
\begin{center}
\includegraphics[width=10cm]{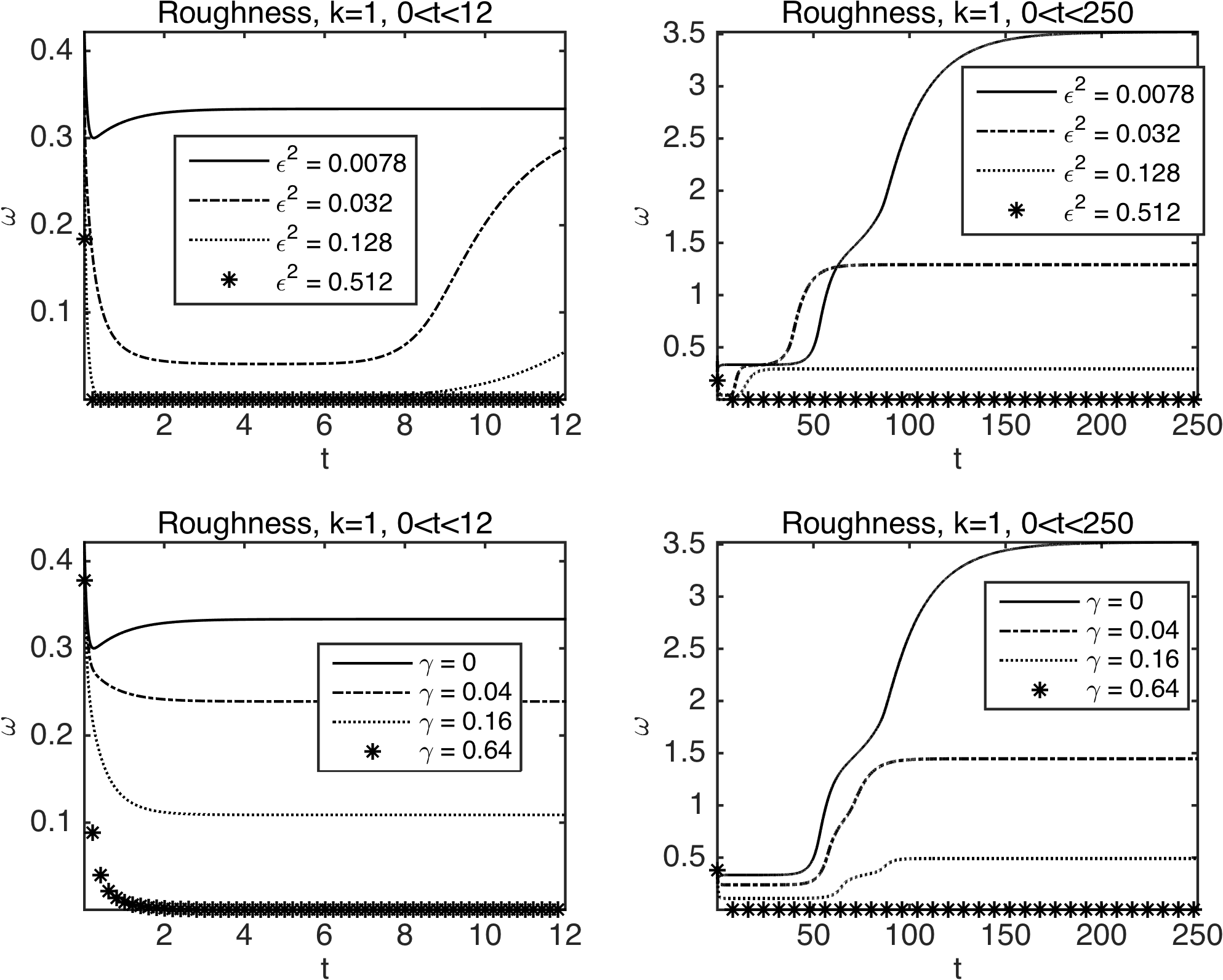}
\caption{Effect of small perturbations of $\epsilon^2$  and $\gamma$ on roughness for Example 1, with $k=1$.} \label{fig:E1-perturb}
\end{center}
\end{figure}

\begin{figure}[ht]
\begin{center}\includegraphics[width=11.5cm]{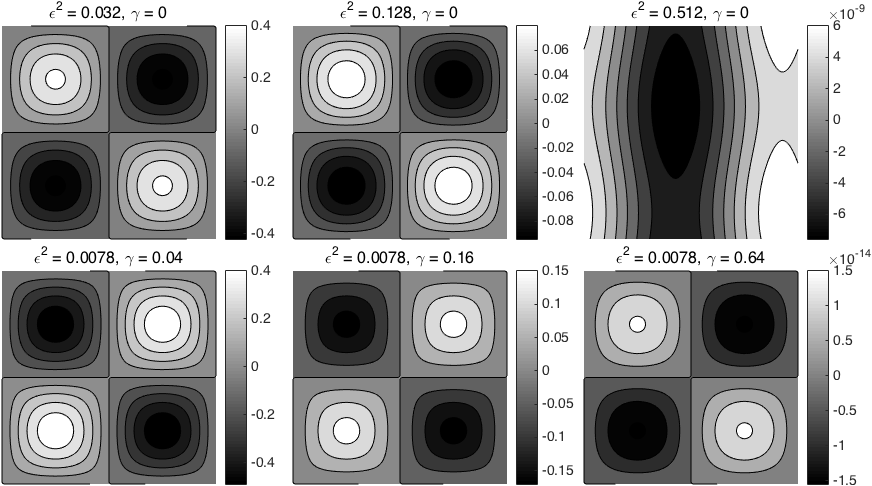}
\caption{Effect of small perturbations of $\epsilon^2$  and $\gamma$ on surface patterns for Example 1, with $k=1$.
All solutions are drawn at time $t=250$.} \label{fig:E1-perturb-surface}
\end{center}
\end{figure}

We also performed perturbation tests on parameters $\gamma$ and $\epsilon^2$.
Only the perturbation for $k=1$ is reported here since the test results are similar for $k=2$ and $k=3$.
The results are given in Figures \ref{fig:E1-perturb} and \ref{fig:E1-perturb-surface},
from which we draw the following conclusions:
\begin{enumerate}
\item Changing $\epsilon^2$ affects both the magnitude of $\omega$ and the time to reach the next ``flat" stage.
A smaller $\epsilon^2$ gives larger surface roughness at steady state, while slows down the evolution. This is reasonable as
$\epsilon^2 \Delta^2 h$ is the highest-order leading dissipation term in Equation (\ref{eq:diffeq}).
When $\epsilon^2$ is set to $0.512$, the fourth-order dissipation term dominates the nonlinear ES effect,
and the surface evolution behaves like a normal fourth-order dissipation, i.e., the roughness quickly drops to $0$ and stays there, as shown in
Figure \ref{fig:E1-perturb}. Besides, from Figure \ref{fig:E1-perturb-surface} one can see that the magnitude of the steady state solution $h$
gradually drops to $0$ as $\epsilon^2$ increases.
\item Changing $\gamma$  affects the magnitude of $\omega$ and slightly affects the pace of the evolution.
Again, this is reasonable as $-\gamma\Delta h$ is the secondary dissipation term in Equation (\ref{eq:diffeq}).
Recall that in subsection \ref{sec:energyfunctional}, we have drawn the conclusion that
when $\gamma$ is large enough, the energy functional $E$ is a convex functional of $\nabla h$ and hence has a unique global minimizer $h=0$.
This has been observed when we increase $\gamma$ to $0.64$ in the perturbation test. In this case, the roughness drops quickly
towards $0$, which indicates that the stabilizing downward funneling current is dominant and the thin film growth
becomes a simple dissipative process. Surface patterns of the steady state solution given in Figure \ref{fig:E1-perturb-surface} also show that
$h$ is nearly $0$ for $\gamma = 0.64$. Again, note that the magnitude of the steady state solution $h$
gradually drops to $0$ as $\gamma$ increases.
\item Another noticeable fact is that, the sign of the steady state solution can get reverted when perturbing $\epsilon^2$,
as shown by comparing the first row of Figure \ref{fig:E1-perturb-surface}  with the steady state solution for $\epsilon^2 = 0.0078$ in Figure \ref{fig:E1-surfacek1}.
However, changing $\gamma$ seems to only affect the magnitude of the steady state solution.
\end{enumerate}

\subsection{Example 2}
In the second example, we pick an initial condition with high frequency:
$$
h_0 = 0.01*(\sin 30x \sin 20y + \sin 50x \sin 50y ).
$$
Several stages of the surface evolution are reported in Figures \ref{fig:E2-surfacek1}-\ref{fig:E2-surfacek3} for $k=1,2,3$.
All solutions converge to the same type of steady state solution that is different from the one for Examples 1,
which  indicates that the steady state solution is not unique.
Again, the case $k=3$ suffers from numerical round-off errors.
The evolution of roughness $\omega$ and energy $E$ is reported in Figure \ref{fig:E2-roughnessenerygy}.

\begin{figure}[ht]
\begin{center}
\includegraphics[width=11.5cm]{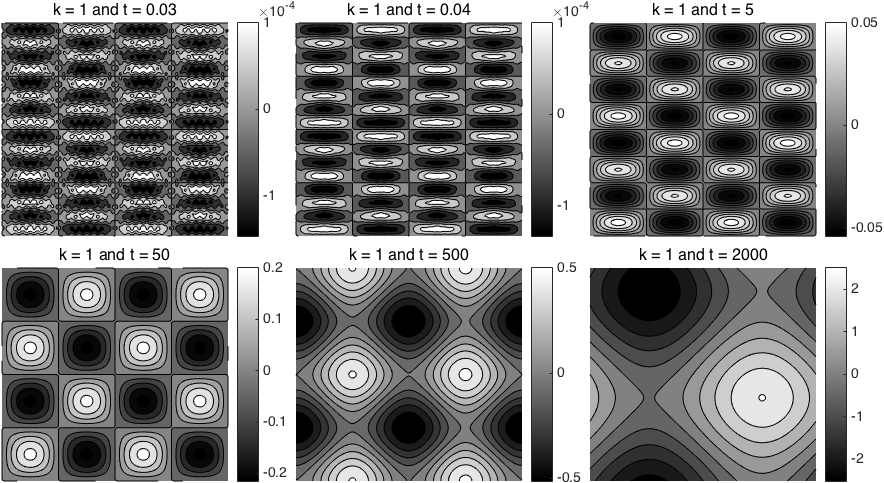}
\caption{Evolution of surface pattern for Example 2, with $k=1$.}  \label{fig:E2-surfacek1}
\end{center}
\end{figure}

\begin{figure}[ht]
\begin{center}
\includegraphics[width=11.5cm]{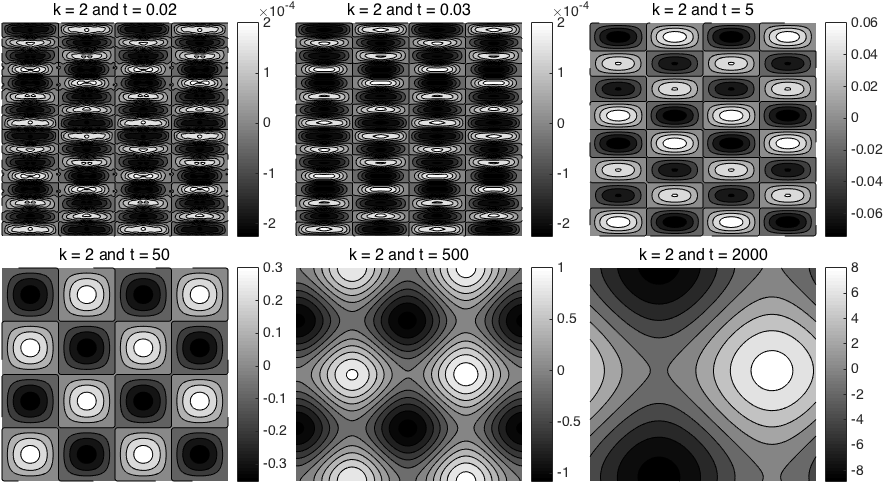}
\caption{Evolution of surface pattern for Example 2, with $k=2$.}  \label{fig:E2-surfacek2}
\end{center}
\end{figure}

\begin{figure}[ht]
\begin{center}
\includegraphics[width=11.5cm]{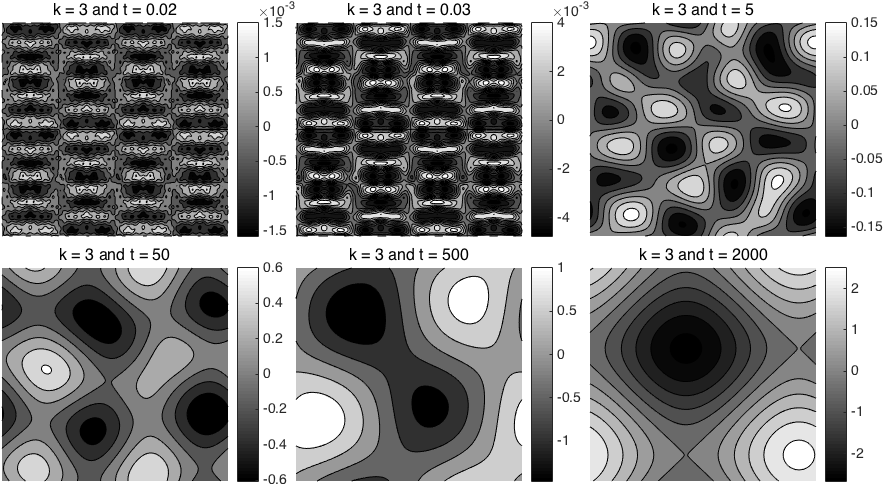}
\caption{Evolution of surface pattern for Example 2, with $k=3$.}  \label{fig:E2-surfacek3}
\end{center}
\end{figure}

 \begin{figure}[ht]
\begin{center}
\includegraphics[width=10cm]{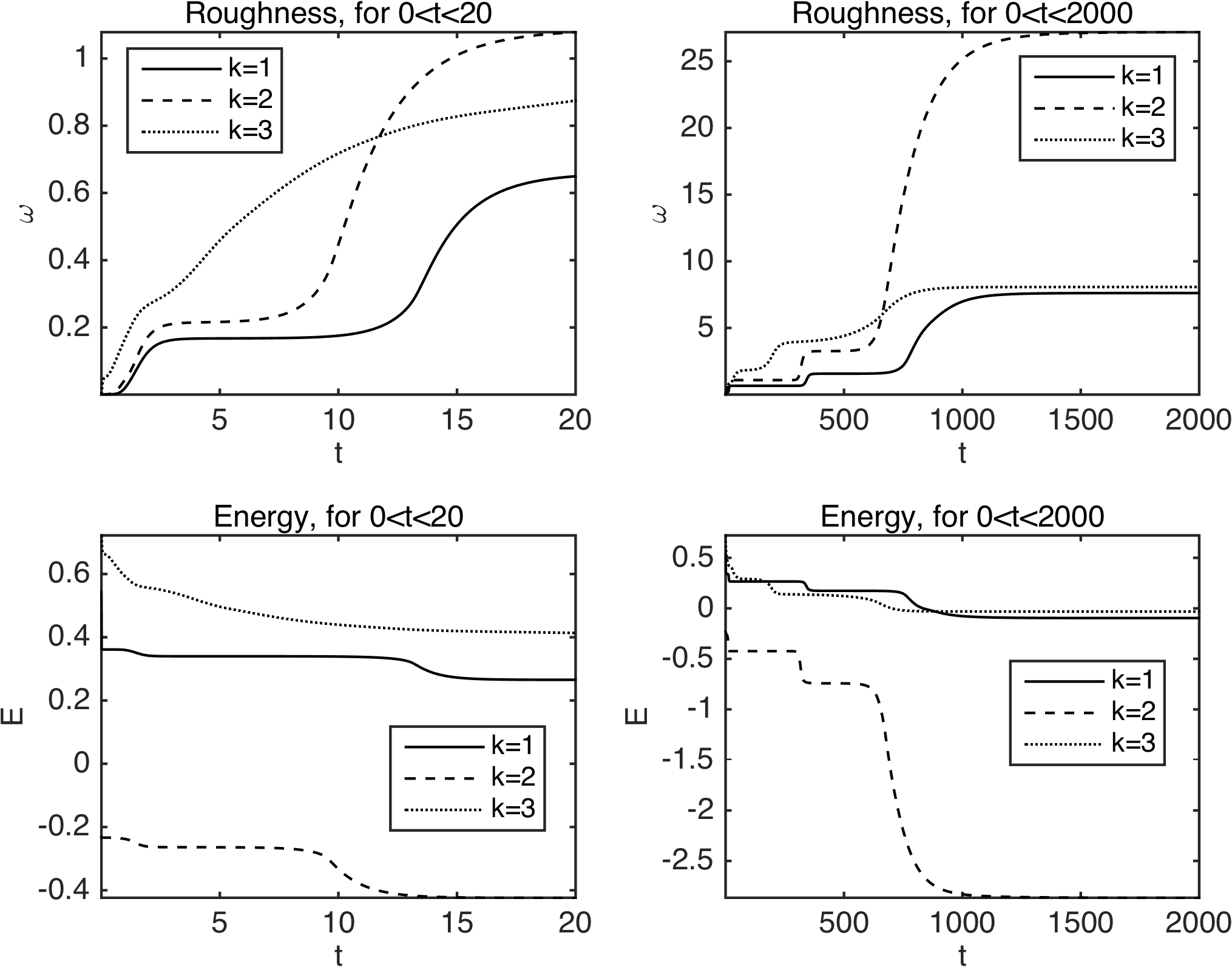}
\caption{Evolution of roughness $\omega$ and energy $E$ for Example 2. Left: early evolution; Right: entire evolution.} \label{fig:E2-roughnessenerygy}
\end{center}
\end{figure}

What is interesting about this example is the fast reduction of frequency (smoothing effect) in the beginning of the evolution.
The initial condition has a high wave number that almost reaches the largest resolution of a $128\times 128$ computational grid.
The surface plot in Figures \ref{fig:E2-surfacek1}-\ref{fig:E2-surfacek3} has to start from $t=0.02$ or $t=0.03$, because the surface plot at $t=0.01$
looks completely ``black" due to its high frequency components. Within a few time steps, all these high oscillation
parts are quickly smoothed out. This is obviously the effect of the fourth order dissipation term $\epsilon^2\Delta^2 h$.

After the initial smoothing process, the magnitude of the solution undergoes a dramatic increase. For example, in Figure \ref{fig:E2-surfacek1},
the magnitude of the solution increases from $10^{-4}$ at $t=0.04$ to $0.05$ at $t=5$ and eventually to $2$ at $t=2000$,
which indicates a typical island-forming or coarsening process. Moreover, the wave number of the solution keeps dropping as $t$ increases,
until it reaches the steady state solution.

Finally, we point out that although the case $k=3$ keeps suffering from round-off errors, it eventually converges to a same type of
steady-state solution as the one for $k=1$ or $k=2$, again with a phase shift.

\subsection{Example 3}
In the third example, we pick
$$
h_0 = \sin 2x \cos 3y.
$$
The purpose of this example is to test a relatively smooth initial data with large magnitude.

\begin{figure}[ht]
\begin{center}
\includegraphics[width=10cm]{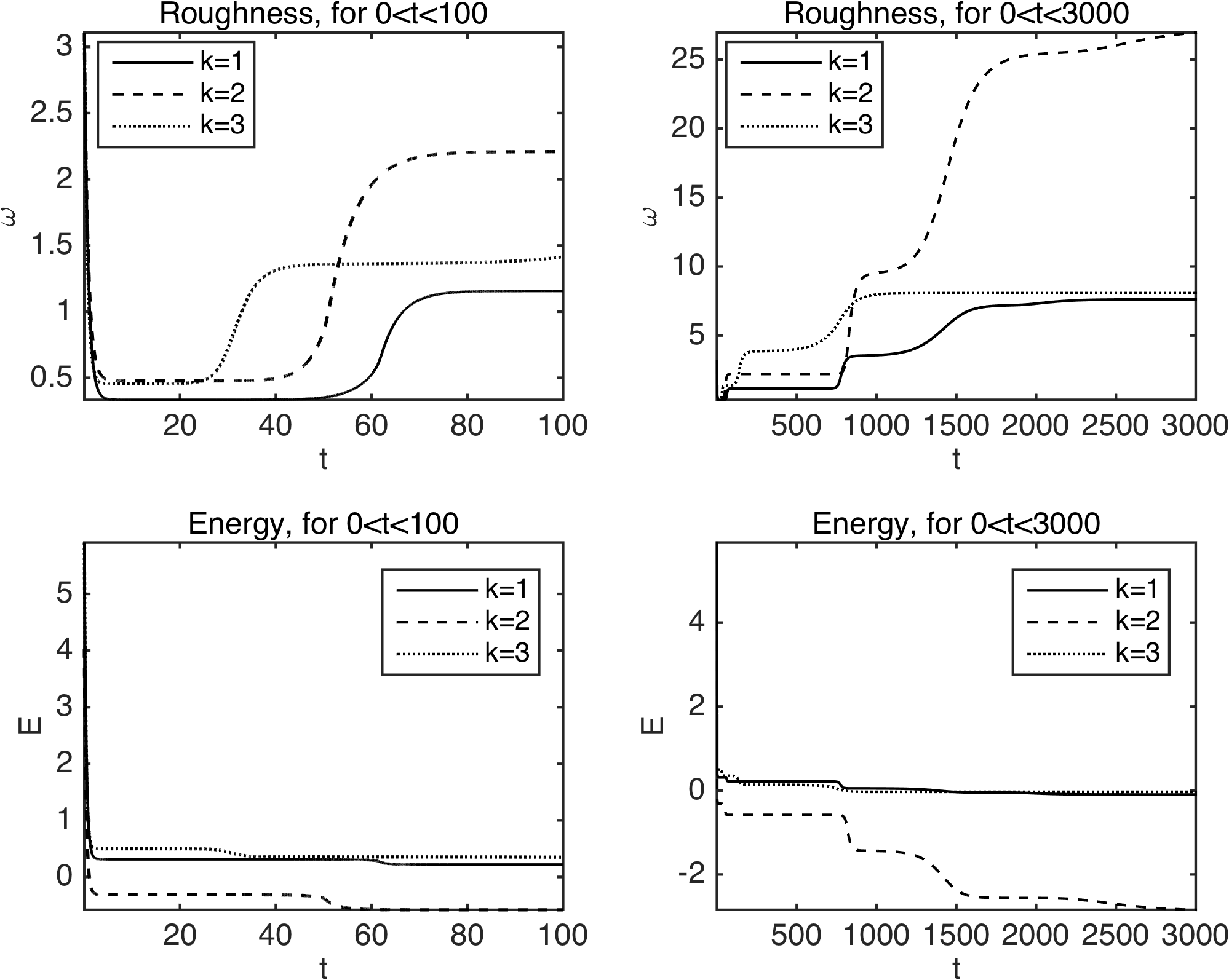}
\caption{Evolution of roughness $\omega$ and energy $E$ for Example 3.} \label{fig:E3-roughnessenerygy}
\end{center}
\end{figure}

\begin{figure}[ht]
\begin{center}
\includegraphics[width=14cm]{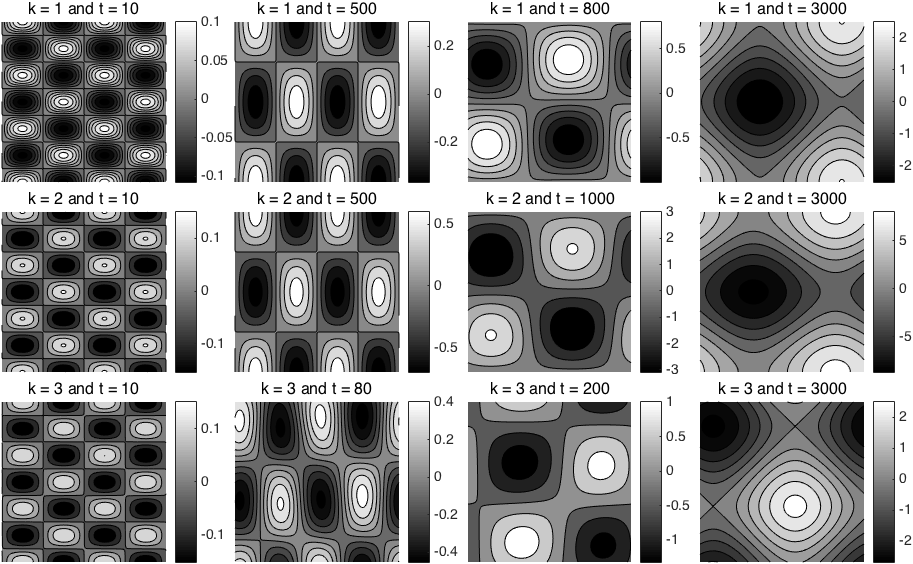}
\caption{Evolution of surface pattern for Example 3.}  \label{fig:E3-surfacek1}
\end{center}
\end{figure}

Results of Example 3 are reported in Figures \ref{fig:E3-roughnessenerygy} and \ref{fig:E3-surfacek1}.
Here we point out the sharp drop of both roughness and energy in the beginning of the evolution, as seen from Figure \ref{fig:E3-roughnessenerygy}.
Indeed, comparing to other examples, Example 3 start from a large $\omega$, which is probably the reason of the sharp drop.
This also agrees with the ``rough-smooth-rough" pattern analyzed in \cite{LL03}.

Again, in Example 3, we found that the solution for $k=1$ and $k=2$ converges to exactly the same steady-state pattern, with the magnitude for $k=2$
larger then for $k=1$; while the solution $k=3$ converges to a slightly different pattern with a phase shift but almost the same magnitude as for $k=1$.
This phenomenon has also been observed for examples 1 and 2.
By examining the roughness and the energy history, one can also see that towards the steady state solution, the roughness and energy curves
for $k=1$ and $k=3$ tend to stay close while the curves for $k=2$ are away from them.

\subsection{Example 4}
In the fourth example, we pick
$$
h_0 = 0.01( \sin 3x \sin 2y + \cos 50x \cos 100y).
$$
which is a combination of a low frequency part with a high frequency part.

\begin{figure}[ht]
\begin{center}
\includegraphics[width=10cm]{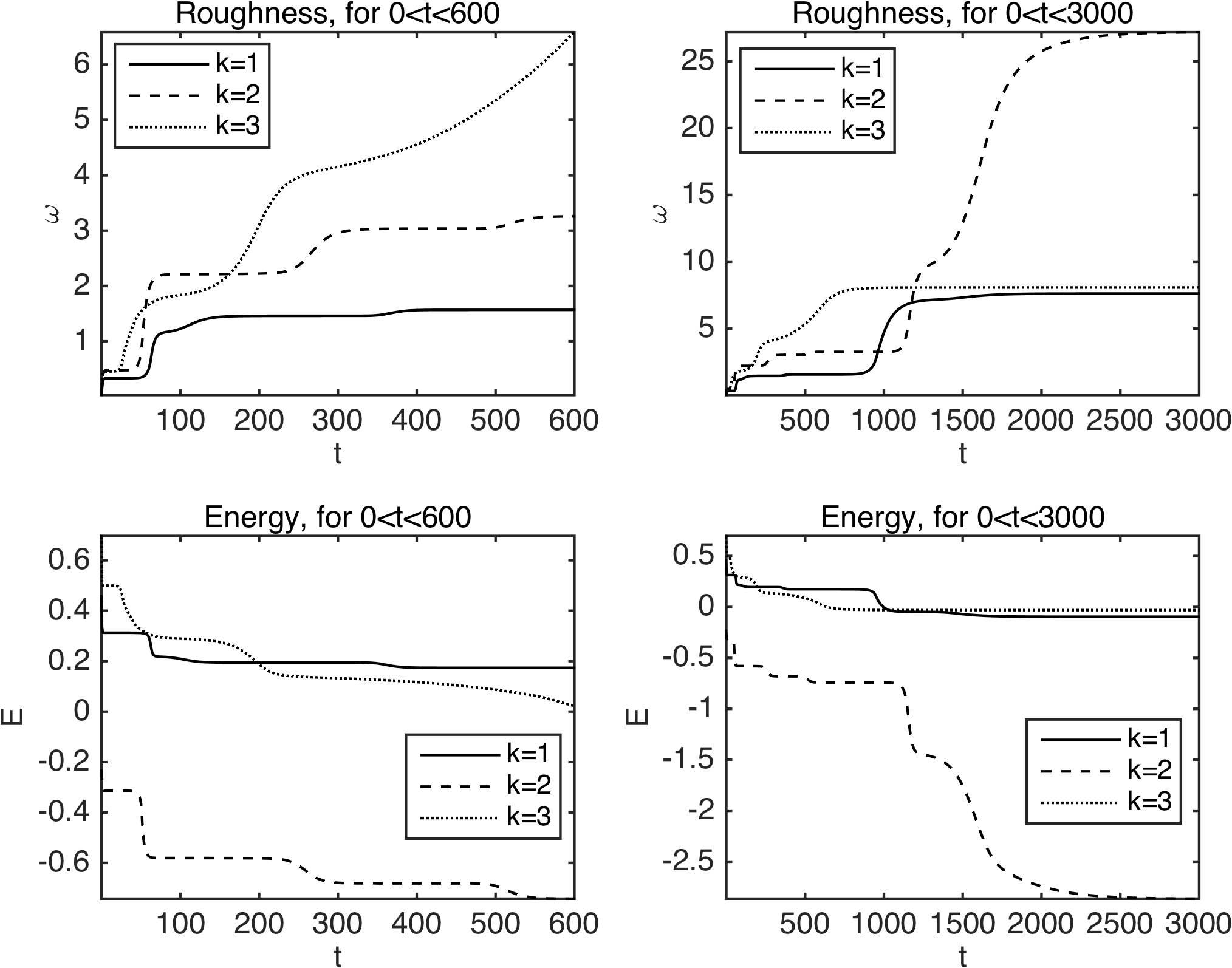}
\caption{Evolution of roughness and energy for Example 4. Left: early evolution; Right: entire evolution.} \label{fig:E4-roughnessenerygy}
\end{center}
\end{figure}

\begin{figure}[ht]
\begin{center}
\includegraphics[width=11.5cm]{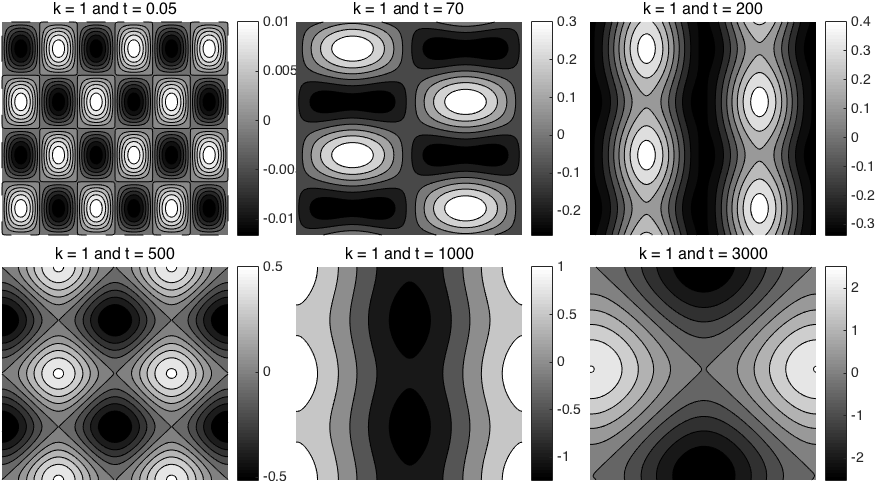}
\caption{Evolution of surface pattern for Example 4, with $k=1$.}  \label{fig:E4-surfacek1}
\end{center}
\end{figure}

\begin{figure}[ht]
\begin{center}
\includegraphics[width=11.5cm]{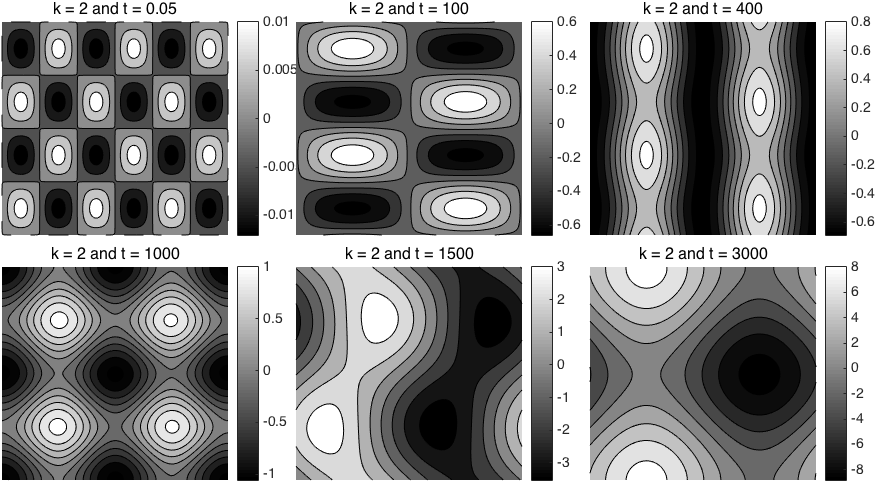}
\caption{Evolution of surface pattern for Example 4, with $k=2$.}  \label{fig:E4-surfacek2}
\end{center}
\end{figure}
\begin{figure}[ht]
\begin{center}
\includegraphics[width=11.5cm]{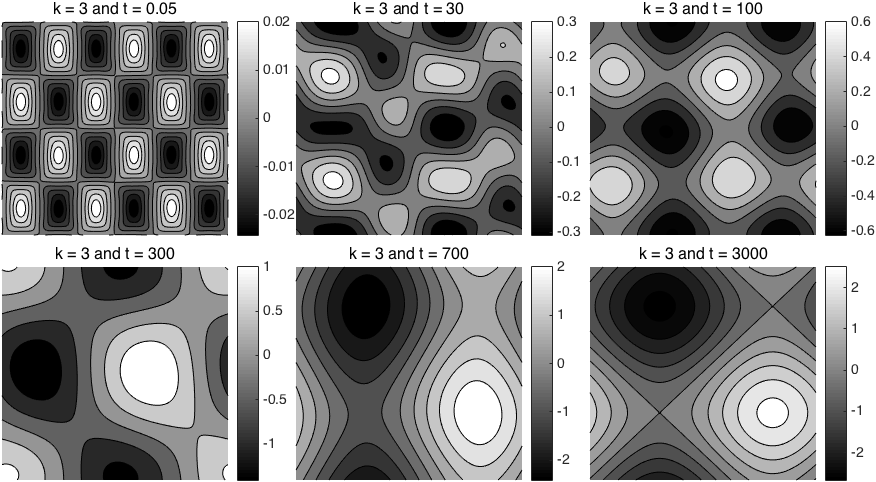}
\caption{Evolution of surface pattern for Example 4, with $k=3$.}  \label{fig:E4-surfacek3}
\end{center}
\end{figure}

Example 4 has the richest evolution process, i.e., the largest amount of ``flat" stages,
 among all examples presented in this paper.
 Moreover, it is the only example we have found so far such that there is a significant phase shift between the steady-state solutions for $k=1$ and $k=2$.
 Results for Example 4 are reported in Figures \ref{fig:E4-roughnessenerygy}-\ref{fig:E4-surfacek3}.

\subsection{Example 5}
In the fifth and the last example, we pick a completely random initial condition with values in $[-0.5,\, 0.5]$.
The initial condition was generated in Matlab using {\tt rand} and then saved in a file, in order to make sure that all tests start from the same initial condition
instead of another random generation.

\begin{figure}[ht]
\begin{center}
\includegraphics[width=10cm]{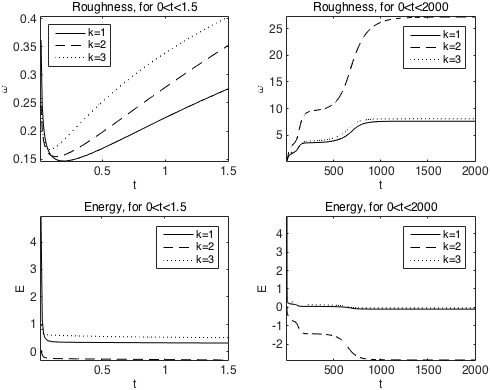}
\caption{Evolution of roughness and energy for Example 5. Left: early evolution; Right: entire evolution.} \label{fig:E5-roughnessenerygy}
\end{center}
\end{figure}

\begin{figure}[ht]
\begin{center}
\includegraphics[width=14cm]{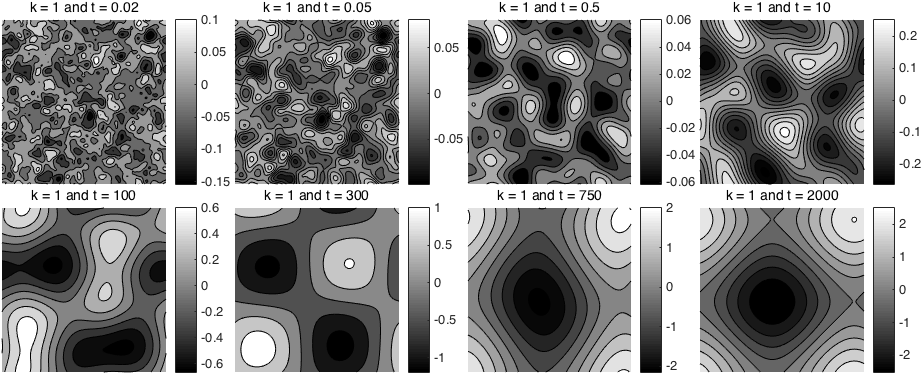}
\caption{Evolution of surface pattern for Example 5, with $k=1$.}  \label{fig:E5-surfacek1}
\end{center}
\end{figure}

\begin{figure}[ht]
\begin{center}
\includegraphics[width=14cm]{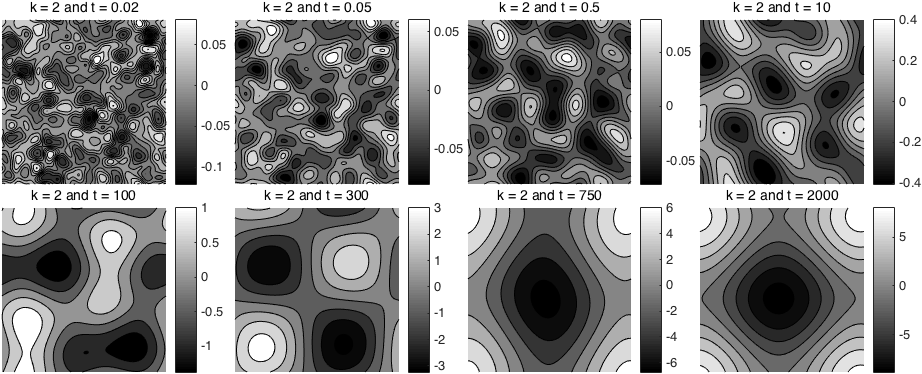}
\caption{Evolution of surface pattern for Example 5, with $k=2$.}  \label{fig:E5-surfacek2}
\end{center}
\end{figure}
\begin{figure}[ht]
\begin{center}
\includegraphics[width=14cm]{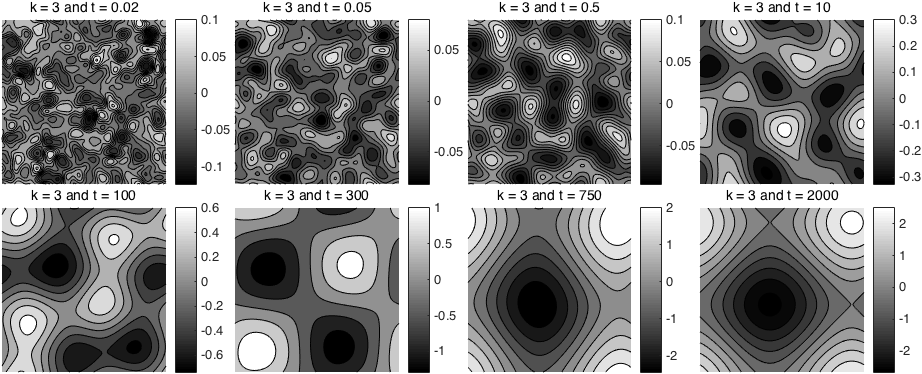}
\caption{Evolution of surface pattern for Example 5, with $k=3$.}  \label{fig:E5-surfacek3}
\end{center}
\end{figure}

Evolution of Example 5 is reported in Figures \ref{fig:E5-roughnessenerygy}-\ref{fig:E5-surfacek3}.
An obvious ``rough-smooth-rough'' pattern is observed in the beginning of the evolution, as shown in
Figure \ref{fig:E5-roughnessenerygy}. Correspondingly, one can see how the ``rough'' random initial condition
is smoothed out in Figures \ref{fig:E5-surfacek1}-\ref{fig:E5-surfacek3}.

More interestingly, we found that for all $k=1,2,3$, the solution goes through almost identical evolution stages,
as shown in  Figures \ref{fig:E5-surfacek1}-\ref{fig:E5-surfacek3}. Such a high similarity has not been observed in previous examples.

%

\appendix
\section{Proof of Theorem \ref{th3}} \label{appendix}
We first introduce a few notations.
Note that $a_+(\cdot,\cdot)$ is a bilinear form and is coercive on $H_{per}^2(\Omega)$. This allows us to define an $a_+$-projection $P_N^+$
from $H_{per}^2(\Omega)$ to $H_N$ by
$$
a_+(P_N^+ v, \phi) = a_+(v,\phi) \qquad \textrm{for all }\phi\in H_N.
$$
It is standard to show that
$$
\|(I-P_N^+) \phi\|_{H^s(\Omega)} \le C_A N^{-(m-s)} \|\phi\|_{H^m(\Omega)} \qquad\textrm{for all } \phi\in H_{per}^m(\Omega) \textrm{ with } s=0,1,2 \textrm{ and }m\ge 2,
$$
where $C_A$ is a positive constant independent of $\phi$ but might depend on $\epsilon$.

For simplicity, denote $h^i=h(\cdot, t_i)$ for $i=0,1,\ldots, M$.
Define
$$
e^i = P_N^+ h^i - h_N^i,\qquad \rho^i = h^i - P_N^+ h^i,
$$
then one has $\underline{h}^i = e^i+\rho^i$.
We further denote
$$
\theta^i = \partial_t h^i - \frac{h^i - h^{i-1}}{\delta t}.
$$
By Taylor expansion and the Schwarz inequality, one has
$$
\|\theta^i\|^2 = \| \frac{1}{\delta t} \int_{t_{i-1}}^{t_i} (t-t_{i-1}) h_{tt} (\cdot, t)\, dt\|^2
\le \frac{\delta t}{3} \int_{t_{i-1}}^{t_i} \|h_{tt}(\cdot,t)\|^2 \, dt.
$$

Subtracting Equation (\ref{eq:scheme}) from Equation (\ref{eq:weakeq}) gives
$$
\langle \theta^i, \phi\rangle +\frac{1}{\delta t}\langle \underline{h}^i - \underline{h}^{i-1},\phi\rangle
  + a_+(\underline{h}^i,\phi) = a_- (h_N^{i-1}, \phi) - a_-(h^i,\phi)\qquad\textrm{for all }\phi\in H_N.
$$
Here we have used the fact that $a_+(\cdot,\cdot)$ is a bilinear form.
The above equation can be further rewritten into
$$
\frac{1}{\delta t}\langle e^i - e^{i-1},\phi\rangle  + a_+(e^i,\phi)
= -\langle \theta^i, \phi\rangle - \frac{1}{\delta t}\langle \rho^i - \rho^{i-1},\phi\rangle
+ a_- (h_N^{i-1}, \phi) - a_-(h^i,\phi).
$$
By setting $\phi = 2\delta t \,e^i$, one gets
$$
\begin{aligned}
 &\|e^i\|^2 - \|e^{i-1}\|^2 + \|e^i-e^{i-1}\|^2 + 2\delta t \bigg((\chi+\gamma)\|\nabla e^i\|^2 + \epsilon^2\|\Delta e^i\|^2\bigg) \\
=& -2\delta t \langle \theta^i, e^i\rangle - 2\langle \rho^i - \rho^{i-1}, e^i\rangle
   + 2\delta t \bigg(a_- (h_N^{i-1}, e^i) - a_-(h^i,e^i)\bigg) \\
\triangleq& I_1 + I_2 + I_3 .
\end{aligned}
$$

By the property of $\theta^i$, one has for any constant $C_1$
$$
I_1 \le  \frac{1}{C_1} \delta t \|\theta^i\|^2 + C_1 \delta t \|e^i\|^2
 \le \frac{\delta t^2}{3C_1} \int_{t_{i-1}}^{t_i} \|h_{tt}(\cdot,t)\|^2 \, dt + C_1\delta t \|e^i\|^2.
$$
And for any constant $C_2$,
$$
\begin{aligned}
I_2 &= -2 \langle \int_{t_{i-1}}^{t_i} (I-P_N^+) h_t\, dt, e^i\rangle \le 2 \|\int_{t_{i-1}}^{t_i} (I-P_N^+) u_t\, dt\|\, \|e^i\| \\
  & \le 2\bigg( \delta t \int_{t_{i-1}}^{t_i} \|(I-P_N^+) h_t\|^2\, dt \bigg)^{1/2} \|e^i\| \\
  &\le \frac{1}{C_2}\int_{t_{i-1}}^{t_i} \|(I-P_N^+) h_t\|^2\, dt + C_2\delta t \|e^i\|^2 \\
  &\le \frac{C_A^2N^{-2m_1}}{C_2} \int_{t_{i-1}}^{t_i} \|h_t\|^2_{H^{m_1}(\Omega)} \, dt + C_2\delta t \|e^i\|^2 .
\end{aligned}
$$

By Lemma 2.7, one has $C_P \|\nabla \phi\| \le \|\Delta \phi\|$ for all $\phi\in H_{per}^2(\Omega)$,
where $C_P$ is the coefficient related to the Poincar\'{e} inequality. Denote by
$C_0 = \max\{\chi+\gamma, \, C_P^2\epsilon^2\}$, which is to protect against the case when $\chi+\gamma = 0$.
Then one has
$$
C_0 \|\nabla \phi\|^2 \le (\chi+\gamma)\|\nabla \phi\|^2 + \epsilon^2\|\Delta \phi\|^2.
$$

By Corollary \ref{cor:GkBounds}, There exists a positive constant $C_G$ such that $|\nabla_F^2 G_{-} (\vm{m})| \le C_G$ for all $\vm{m}\in \mathbb{R}^2$.
Therefore
$$
\begin{aligned}
I_3 &= 2\delta t\langle \nabla_F G_-(\nabla h_N^{i-1}) - \nabla_FG_- (\nabla h^i), \nabla e^i \rangle \le 2\delta t C_G \|\nabla (h_N^{i-1}-h^i)\| \, \|\nabla e^i\|\\[2mm]
&\le \frac{\delta t C_G^2}{C_0} \|\nabla (h_N^{i-1}-h^i)\|^2 + \delta t C_0 \|\nabla e^i\|^2 \\[2mm]
&\le \frac{3\delta t C_G^2}{C_0} \bigg(\|\nabla e^{i-1}\|^2 + \|\nabla \rho^{i-1}\|^2 + \|\nabla (h^i-h^{i-1})\|^2 \bigg)
  + \delta t \bigg( (\chi+\gamma)\|\nabla e^i\|^2 + \epsilon^2\|\Delta e^i\|^2\bigg)
\end{aligned}
$$
Next, note that
\begin{equation} \label{eq:I3-1}
\frac{3\delta t C_G^2}{C_0} \|\nabla e^{i-1}\|^2 \le \frac{3\delta t C_G^2}{C_0}\|e^{i-1}\|\,\|\Delta e^{i-1}\|
\le \frac{9 \delta t C_G^4}{2C_0^2\epsilon^2} \|e^{i-1}\|^2 + \frac{\delta t \epsilon^2}{2} \|\Delta e^{i-1}\|^2
\end{equation}
and
\begin{equation} \label{eq:I3-2}
\frac{3\delta t C_G^2}{C_0} \|\nabla \rho^{i-1}\|^2 = \frac{3\delta t C_G^2}{C_0} \|\nabla (I-P_N^+) h^{i-1}\|^2
\le \frac{3\delta t C_G^2 C_A^2}{C_0} N^{-2(m_2-1)} \|h^{i-1}\|_{H^{m_2}(\Omega)}^2.
\end{equation}
and
$$
\frac{3\delta t C_G^2}{C_0} \|\nabla (h^i-h^{i-1})\|^2 = \frac{3\delta t C_G^2}{C_0} \|\int_{t_{i-1}}^{t_i} \nabla h_t \, dt\|^2
\le \frac{3\delta t^2 C_G^2}{C_0} \int_{t_{i-1}}^{t_i} \|\nabla h_t\|^2 \, dt.
$$
When $i=1$, one shall replace the estimates in (\ref{eq:I3-1}) and (\ref{eq:I3-2}) by a combined term
$$
\frac{3\delta t C_G^2}{C_0} \|\nabla (h_N^0 - h^0)\|^2
= \frac{3\delta t C_G^2}{C_0} \|\nabla e^0\|^2 .
$$

Choose $C_1$ and $C_2$ to ensure $(C_1+C_2)\delta t \le \frac{1}{2}$.
Combine all the above, and sum up for $i=1,\ldots,n$, use the definition of $e^0$ and Lemma \ref{lem:PNapproximation}, one has
$$
\begin{aligned}
\frac{1}{2}\|e^n\|^2 + &\delta t \sum_{i=1}^n \bigg((\chi+\gamma) \|\nabla e^i\|^2 + \frac{\epsilon^2}{2}\|\Delta e^i\|^2 \bigg)
\le \|e^0\|^2 + \frac{3\delta t C_G^2}{C_0} \|\nabla e^0\|^2 \\
 & + C \delta t^2\bigg(\|h_{tt}\|_{L^2(0, t_n; L^2(\Omega))}^2 + \|h_t\|_{L^2(0, t_n; H^1(\Omega))}^2  \bigg)  \\
&+ C N^{-2m_1} \|h_t\|_{L^2(0, t_n; H^{m_1}(\Omega))}^2
+ C t_n N^{-2(m_2-1)} \|h\|_{L^{\infty}(0, t_{n-1}; H^{m_2}(\Omega))}^2 + C\delta t \sum_{i=1}^{n-1} \|e^i\|^2 \\
\le\, & C (\delta t^2 + N^{-2m_1} + N^{-2(m_2-1)})  + C\delta t \sum_{i=1}^{n-1} \|e^i\|^2,
\end{aligned}
$$
where $C$ is a general constant that may depend on $C_1$, $C_2$, $C_P$, $C_A$, $C_G$, $\gamma$, $\chi$, $\epsilon$,
but not on $\delta t$ or $N$.
Then the result follows from the Gronwall's inequality and the triangle inequality.

\medskip

\textbf{Acknowledgments:} We are heartily grateful to Dr. Xiaoming Wang. The topic was suggested by him and some difficulties were overcome with his helps.
Yanqiu Wang thanks the Key
Laboratory of Mathematics for Nonlinear Sciences, Fudan University, for the support during her visit.


\end{document}